\documentclass{amsart}

\usepackage{geometry}

\title[KK uniqueness for semifinite factors]{A note on 
proper asymptotic uniqueness for  semifinite
factors}

\author{P. W. Ng}

\address{Department of Mathematics\\  University of Louisiana at Lafayette\\
217 Maxim Doucet Hall\\   1401 Johnston St.\\
Lafayette, Louisiana\\
70503\\  USA}

\email{png@louisiana.edu}

\author{Cangyuan Wang}

\address{School of Mathematical Sciences\\
	Ocean University of China\\
	238 Songling Road, Laoshan District\\
	Qingdao, Shandong\\
266100\\ 
P. R. China}

\email{cyw@ouc.edu.cn}

\newtheorem{thm}{Theorem}[section]   
\newtheorem{df}[thm]{Definition} 
\newtheorem{prop}[thm]{Proposition}  
  
\newtheorem{lem}[thm]{Lemma}

\newcommand{\A}{\mathcal{A}}
\newcommand{\B}{\mathcal{B}}
\newcommand{\C}{\mathcal{C}}
\newcommand{\D}{\mathcal{D}}
\newcommand{\E}{\mathcal{E}}
\newcommand{\F}{\mathcal{F}}

\newcommand{\K}{\mathcal{K}}

\newcommand{\Mul}{\mathcal{M}}
\newcommand{\M}{\mathbb{M}}  

\newcommand{\T}{\mathcal{T}}
\newcommand{\Ss}{\mathcal{T}}

\newcommand{\hh}{\mathcal{H}}

\numberwithin{equation}{section}

\begin{document}

\maketitle

\begin{abstract}
Let $\A$ be a separable nuclear C*-algebra, and let $\Mul$ be a 
semifinite von Neumann factor with separable predual.
Let $\phi, \psi: \A \rightarrow \Mul$ be essential trivial extensions 
with $\phi(a) - \psi(a) \in \K_{\Mul}$ for all $a \in \A$ such 
that either 
both $\phi$ and $\psi$ (and hence $\A$) are unital or both $\phi$ and $\psi$
have large complement.

Then 
$\phi$ and $\psi$ are properly asymptotically
unitarily equivalent if and only if $[\phi, \psi]_{CS} = 0$ in
$KK(\A, \C(S \K_{\Mul}))$. 
\end{abstract}

\section{Introduction}

In their groundbreaking paper \cite{BDF1}, Brown, Douglas and Fillmore
classified all essentially normal operators using algebraic topological
invariants. In the course of proving functorial properties
of their functor, they introduced the notion of the
\emph{essential codimension} $[P:Q]$ for a pair of projections
$P, Q \in \mathbb{B}(l_2)$          
with $P - Q \in \K$ (where $\K$ is the algebra of compact operators
over a separable infinite dimensional Hilbert space $l_2$), and they
 showed that
\begin{equation} \label{equ:BDFECTh}
[P:Q] = 0 \makebox{  if and only if } \exists   
\makebox{ a unitary  }U \in \mathbb{C}1 +  
\K \makebox{ for which } P = U Q U^*.
\end{equation}   
(See, for example, \cite{LoreauxNg} and 
\cite{LoreauxNgSutradhar1}.) 
These notions have turned out to be quite fruitful.
For examples, the notion of essential codimension
has led to an explanation for the mysterious integers appearing
in Kadison's Pythagorean theorem and other Schur--Horn type 
results (e.g., \cite{LoreauxArveson}, \cite{KadisonPyth1}, 
\cite{KadisonPyth2}, \cite{KaftalLoreaux}), and has had applications
to the study of spectral flow and index theorems (e.g., 
see \cite{PhillipsEtAl}). 
Moreover, it turns out that essential codimension is actually a special
case of a $KK^0$ element, and the BDF essential codimension result 
(\ref{equ:BDFECTh}) has nontrivial generalizations that are important
for the stable uniqueness of theorems of classification theory. Among
other things, we have the following result: 
\begin{thm} \label{thm:MultiplierBDFECTh} (See \cite{LeeProperAUE},
\cite{LoreauxNg}, \cite{LoreauxNgSutradhar2}, \cite{DadarlatEilers},
\cite{LinStableAUE}.)  
Let $\A$ be a separable nuclear C*-algebra, and let $\B$ be a separable
stable C*-algebra.
Suppose that the Paschke dual algebra $(\A^+)^d_{\B}$ is $K_1$-injective.

Let $\phi, \psi : \A \rightarrow \Mul(\B)$ be *-monomorphisms with
$\phi(a) - \psi(a) \in \K_{\Mul}$ for all $a \in \A$ such that
either both are unitally absorbing trivial extensions or both are nonunitally
absorbing trivial extensions.

Then 
$\phi$ and
$\psi$ are properly asymptotically unitarily equivalent 
if and only if 
$[\phi, \psi] = 0$ in $KK^0(\A, \B)$   
\end{thm}    

In this note, we seek to 
extend Theorem \ref{thm:MultiplierBDFECTh} to 
the context of a type $II_{\infty}$ von Neumann factor $\Mul$ with separable
predual; i.e., we replace $\Mul(\B)$ with $\Mul$ (see Theorem 
\ref{thm:TotalMainResult}).  Note that if $\K_{\Mul}$
is the Breuer ideal of $\Mul$, then $\Mul$ is actually the multiplier algebra
$\Mul = \Mul(\K_{\Mul})$ 
(see Lemma \ref{lem:MultiplierAlgebraOfSubalgebraOfVNA}).  However, unlike $\B$, $\K_{\Mul}$ is not
even $\sigma$-unital, 
and thus, $KK^0(\A, \K_{\Mul})$ does not
have many  nice properties. 
Thus, we replace $KK^0(\A, \K_{\Mul})$ with $KK^0(\A, \C(S \K_{\Mul}))$,
noting that the corona algebra $\C(S \K_{\Mul})$ is unital.  
Moreover, also because of the non-$\sigma$-unitality of $\K_{\Mul}$, 
a number of results for multiplier algebras need to be established
in this new setting.  On the other hand, $\Mul/\K_{\Mul}$ is simple purely
infinite (which is not always true for $\Mul(\B)/\B$), and as a consequence,
a number of simplifications are also present -- including 
a Voiculescu type
Weyl--von Neumann theorem (see Theorem \ref{thm:VNAAbsorption}) which
results in a cleaner statement, dropping some of the assumptions present
in Theorem \ref{thm:MultiplierBDFECTh}.

We now discuss the contents of the paper. 
In Section \ref{sect:Reduction}, we discuss some preliminary results about
extension theory, as well as 
type $II_{\infty}$ factors.   Some of the main results of this section
concern ways to reduce from $\Mul$ to the case of a multiplier algebra
$\Mul(\B)$ of a separable stable C*-algebra $\B$.  Some of these techniques
are taken from the first author's joint work with others in 
\cite{GiordanoKaftalNg}.  Since this manuscript is not yet
available, we provide
proofs for the convenience of the reader.  
The ``concrete" reduction arguments at the end of Section \ref{sect:Reduction}
are straightforward applications of the abstract reduction arguments presented
earlier in the section.  These arguments are not difficult conceptually, 
but they are
technical and laborious.  
In Section \ref{sect:AWvN}, we 
present a $II_{\infty}$ factor version of the Voiculescu--Weyl--von Neumann
theorem (see Theorem \ref{thm:VNAAbsorption}).  Versions of this are
already available in the literature (see the interesting papers
\cite{LiShenShi}, \cite{HadwinMaShen}; \cite{GiordanoNg}).  We also discuss $K_1$-injectivity of the Paschke
dual algebra in the $II_{\infty}$ setting.      
In Section \ref{sect:MainResult}, we prove our main result, which is
Theorem \ref{thm:TotalMainResult}.
Finally, in the appendix, we provide a short KK computation which is required
for the main argument of this paper.

\section{Preliminaries and reduction arguments}
\label{sect:Reduction}

We begin by briefly introducing some notation.
We refer the reader to \cite{DavidsonBook}, \cite{WeggeOlsen}, 
\cite{BlackadarKTh}. \cite{JensenThomsenBook} and \cite{LinBook}
for basic results in C*-algebras, K theory and KK theory.
 
We let $l_2$ be our notation for a separable infinite dimensional 
Hilbert space, and let $\mathbb{B}(l_2)$ be the 
C*-algebra of all bounded linear operators on $l_2$.  We let 
$\K$ denote the C*-algebra of all compact operators on $l_2$ (so 
$\K \subseteq \mathbb{B}(l_2)$ is a C*-subalgebra).

For a C*-algebra $\B$, 
$\Mul(\B)$ denotes the multiplier algebra of $\B$, and $\C(\B) =_{df} 
\Mul(\B)/\B$ denotes the corona algebra of $\B$.
(E.g., $\Mul(\K) = \mathbb{B}(l_2)$ and $\C(\K) = \mathbb{B}(l_2)/\K$.) 
We let $\pi_{\B} : \Mul(\B) \rightarrow \C(\B)$  denote
 the usual quotient map. 
Often, when the context is clear, we drop the $\B$ and write $\pi$ instead
of $\pi_{\B}$ for this quotient map. See next paragraph for other quotient
map 
notation.

For a semifinite von Neumann factor $\Mul$ (i.e., $\Mul$ is either
$II_{\infty}$ or $\mathbb{B}(l_2)$) with separable predual,
we let $\K_{\Mul}$ denote the Breuer ideal of $\Mul$.  (So 
if $\Mul = \mathbb{B}(l_2)$, then $\K_{\Mul} = \K$.)  
In Lemma \ref{lem:MultiplierAlgebraOfSubalgebraOfVNA}, we will see 
that $\Mul = \Mul(\K_{\Mul})$ (even when $\Mul$ is type $II_{\infty}$).  
For clarity, we sometimes let $\pi_M : \Mul \rightarrow \Mul/\K_{\Mul}$
denote the quotient map.

Finally, for 
any C*-algebra $\D$,  for any $x, y \in \D$, and for any $\epsilon > 0$, we use the notation
$x \approx_{\epsilon} y$ to mean $\| x - y \| < \epsilon$.

\subsubsection{} \label{para:Extensions1} 
We next recall some preliminaries about extensions
(which will primarily be used starting in Lemma \ref{lem:PAsympUEReduction}).
More detailed information for extension theory can be found in 
\cite{WeggeOlsen} (see also \cite{BlackadarKTh} and \cite{LinBook}).
Recall that to an extension of C*-algebras
\begin{equation} \label{equ:Apr120251AM}
  0 \to \B \to \mathcal{E} \to \A \to 0,
\end{equation}
we can associate the \emph{Busby invariant} of the extension, 
which is a *-homomorphism $\phi : \A \to \C(\B)$.
Conversely, given such a *-homomorphism $\phi : \A \to \C(\B)$, one can obtain an extension of the form (\ref{equ:Apr120251AM}) whose Busby invariant is 
$\phi$. 
In fact,
the extension corresponding to a given Busby invariant is unique up to 
strong isomorphism (in the terminology of Blackadar;  
see \cite{BlackadarKTh} section 15.4; see also \cite{WeggeOlsen} Corollary 
3.2.12).
Our results are all invariant under strong isomorphism, and therefore, whenever
we have a *-homomorphism $\phi : \A \to \C(\B)$, we will simply refer to
it  as an extension.
When $\phi$ is injective, then $\phi$ is called an 
\emph{essential extension}.  
This corresponds to $\B$ being an essential ideal of $\E$.

A \emph{trivial extension} $\phi : \A \rightarrow \C(\B)$ is an extension
 for which the short exact sequence (say (\ref{equ:Apr120251AM}))
 is split exact, or equivalently, the Busby invariant factors through the 
quotient map $\pi : \Mul(\B) \to \C(\B)$ via a *-homomorphism $\phi_0 : \A \to \Mul(\B)$ so that $\phi = \pi \circ \phi_0$.
If $\phi : \A \to \C(\B)$ is a trivial (essential) extension, then by a slight 
abuse of terminology, we also refer to the lifting *-homomorphism
$\phi_0 : \A \rightarrow \Mul(\B)$ as a trivial
(resp. essential) extension.
When $\A$ is a unital C*-algebra and $\phi : \A \rightarrow \C(\B)$ 
is a unital *-homomorphism, 
we say that $\phi$ is a \emph{unital extension}.
If $\phi : \A \rightarrow \C(\B)$ is unital and
 trivial (and essential) and some lifting *-homomorphism
 $\phi_0$ is also unital, then
$\phi$ is said to be a \emph{strongly unital trivial (resp. essential)
extension}. In this case,   
by abuse of terminology again, 
we often refer to the *-homomorphism
$\phi_0 : \A \rightarrow \Mul(\B)$ as a strongly
unital trivial (resp. essential) extension.  
     
Next, let $\phi, \psi: \A \rightarrow \Mul(\B)$ be *-homomorphisms.
Then $\phi$ and $\psi$ are \emph{asymptotically unitarily equivalent
modulo $\B$ ($\phi \sim_{asymp, \B} \psi$)} if there exists a norm-continuous
path $\{ u_t \}_{t \in [1, \infty}$ of unitaries in $\Mul(\B)$ such that
for all $a \in \A$, (i) $u_t \phi(a) u_t^* - \psi(a) \in \B$ for all $t$,
and (ii.) $\| u_t \phi(a) u_t^* - \psi(a) \| \rightarrow 0$ as  
$t \rightarrow \infty$.

The following is the central equivalence relation studied in 
this paper:

\begin{df} \label{df:ProperAsympUE}
Let $\A$ and $\B$ be C*-algebras, and let $\phi, \psi : \A \rightarrow \Mul(\B)$ be *-homomorphisms
for which $\phi(a) - \psi(a) \in \B$ for all $a \in \A$.

$\phi$ and $\psi$ are \emph{properly asymptotically unitarily equivalent} ($\phi \sim_{pasymp} \psi$) 
if there exists a norm-continuous path $\{ u_t \}_{t \in [1, \infty)}$, of unitaries in $\mathbb{C}1
+ \B$,
such that for all $a \in \A$,  
\begin{enumerate}
\item $u_t \phi(a) u_t^* - \psi(a) \in \B$ for all $t \in [1, \infty)$, and  
\item $\| u_t \phi(a) u_t^* - \psi(a) \| \rightarrow 0$ as $t \rightarrow \infty$.  
\end{enumerate}
(Note that $\B$ can be non$\sigma$-unital. E.g., $\B = \K_{\Mul}$,
the Breuer ideal of a $II_{\infty}$ factor $\Mul$.)
\end{df}

Finally, for C*-algebras $\A$ and $\C$ with $\C$ unital,
a map $\rho : \A \rightarrow \C$ 
\emph{has large complement} if there exists a projection
$p \in \C$ such that $p \sim 1_{\C}$ and $p \perp \rho(\A)$.\\

The proof of the first result is in \cite{GiordanoKaftalNg}.    
But since this manuscript has not yet appeared, we give the brief proof.  

\begin{lem} \label{lem:MultiplierAlgebraOfSubalgebraOfVNA}
Let $\Mul$ be a von Neumann algebra, and let $\B \subseteq \Mul$ be a C*-subalgebra.

Then $\Mul(\B) 
\cong \{ x \in \overline{\B}^{w*} : x \B, \B x \subseteq  \B \}.$\\ 
\noindent As a consequence, we may realize $\Mul(\B)$ as a C*-subalgebra
of $\Mul$ in the obvious way (i.e., $\Mul(\B) \subseteq \Mul$, which 
extends the inclusion $\B \subseteq \Mul$).  
 
Moreover, if $\Mul$ is a semifinite von Neumann factor with separable predual, and if 
$\K_{\Mul} \subseteq \Mul$ is the Breuer ideal of $\Mul$, then
$\Mul(\K_{\Mul}) = \Mul.$
\end{lem}

\begin{proof}
Let $P \in \Mul$ be the range projection of $\B$. I.e., if $\{ e_{\alpha} \}$
is an approximate unit for $\B$, then $P =_{df} \sup_{\alpha} e_{\alpha}
\in \Mul$.

We can find a Hilbert space $\hh$ where we can
realize $P\Mul P$ as an SOT-closed unital *-subalgebra
$P\Mul P \subseteq \mathbb{B}(\hh)$. (So $P = 1_{\mathbb{B}(\hh)}$.)     
Since $P \in \Mul$ is the range projection of $\B$,
$\B$ acts nondegenerately (as well as faithfully) on $\hh$.  
Hence, $\Mul(\B) = \{ x \in \mathbb{B}(\hh) : x \B, \B x 
\subseteq \B \} \subseteq \mathbb{B}(\hh)$ (e.g., see 
\cite{WeggeOlsen} Definition 2.2.2).  
But for all $x \in \Mul(\B)$, $x$ is in the strict closure
of $\B$.  Thus, since $\B$ sits nondegenerately over $\hh$,
for all $x \in \Mul(\B)$, $x$ is in the SOT-closure of $\B$.
Since $P \Mul P$ is an SOT-closed *-subalgebra of $\mathbb{B}(\hh)$, 
for all $x \in \Mul(\B)$, $x \in \overline{\B}^{w*} \subseteq
 P \Mul P$.  I.e., for all $x \in \Mul(\B)$,
$x$ is in the w*-closure of $\B$ in $\Mul$.

The last statement follows from the previous statements, and from the
fact that if $\Mul$ is a semifinite von Neumann factor with separable
predual, then $\K_{\Mul}$ is w*-dense in $\Mul$.\\ 
\end{proof}

We will need some reduction arguments, again  
from \cite{GiordanoKaftalNg}, which 
is not yet available.

\begin{lem} \label{lem:QuasiCentralSOTApproximateUnit}
Let $\Mul$ be a semifinite factor with separable predual,  let $\Ss
 \subset \Mul$ be a countable
subset, and let $\Ss_K \subset \K_{\Mul}$ be a countable subset.  

Then we can find a sequence $\{ e_n \}$, of increasing positive elements of $\K_{\Mul}$, for which
the following statements are true:
\begin{enumerate}
\item $e_{n+1} e_n = e_n$ for all $n$.  
\item For all $x \in \Ss$, $\| x e_n - e_n x \| \rightarrow 0$, as $n \rightarrow \infty$. 
\item For all $y \in \Ss_K$, $\| y e_n - y \|, \| e_n y - y \| \rightarrow 0$, as 
$n \rightarrow \infty$.   
\item $e_n \rightarrow 1_{\Mul}$ in the weak*-topology.  
\end{enumerate}
\end{lem}   

\begin{proof}
This is a variation of Arveson's quasicentral approximate units
argument (\cite{ArvesonDuke}).

Let $\epsilon > 0$ be given, $x_1, ..., x_n \in \Ss$,  and let $p \in \K_{\Mul}$ be a  projection. 
(Recall that $\K_{\Mul}$ is our notation for the Breuer ideal of $\Mul$, and thus,
$p$ is a finite projection in $\Mul$.) Consider the
net $\{ p_{\alpha} \}_{\alpha \in I}$, consisting
 of all projections in $\K_{\Mul}$ which contain $p$, and ordered by the $\leq$ relation.
(So $p \leq p_{\alpha} \leq p_{\beta} \in Proj(\K_{\Mul})$, for all 
$\alpha, \beta \in I$ with $\alpha \leq \beta$; and the range of 
the net $\{ p_{\alpha} \}_{\alpha \in I}$ is $\{ r \in Proj(\K_{\Mul}) :
p \leq r \}$.)  
Note that $\{ p_{\alpha} \}_{\alpha
\in I}$ 
is an approximate unit for $\K_{\Mul}$.  
Consider $E \subset (\K_{\Mul})^n$ which is given by
$E =_{df} Conv(\{ ([x_1, p_{\alpha}], [x_2, p_{\alpha}], ..., [x_n, p_{\alpha}]) :  \alpha \in I\})$,
where $[x_j, p_{\alpha}] = x_j p_{\alpha} - p_{\alpha} x_j$ for all $1 \leq j \leq n$ and
$\alpha \in I$.   
(So $E$ is the convex hull of certain $n$-tuples of commutators.)  
Now consider the (norm-) closure $\overline{E}$, which 
also a convex set. (So $\overline{E}$ is the closure of $E$
in $(\K_{\Mul})^n$, with the
norm topology).

Suppose, to the contrary, that $0 \notin \overline{E}$.   
By the Hahn--Banach separation theorem, we can find a norm one linear functional  $\rho \in ((\K_{\Mul})^n)^*$ 
and $\delta > 0$ such that for all $y \in \overline{E}$,
$$0 < \delta \leq \rho(y).$$ 
But since $\{ p_{\alpha} \}_{\alpha \in I}$ is an approximate unit 
for $\K_{\Mul}$,  $\{ \bigoplus^n p_{\alpha} \}_{\alpha \in I}$ is an
approximate unit for $(\K_{\Mul})^n$ (where $\bigoplus^n p_{\alpha}
= p_{\alpha} \oplus p_{\alpha} \oplus ... \oplus p_{\alpha}$ (direct sum 
$n$ times)), 
and hence,   
$$\rho(([x_1, p_{\alpha}], [x_2, p_{\alpha}], ..., [x_n, p_{\alpha}])) \rightarrow 0.$$
This is a contradiction.
Hence, $0 \in \overline{E}$.

Since $\epsilon$, $x_1, ..., x_n$ and $p$ were arbitrary, we have that for every $\epsilon > 0$,
for every $x_1, ..., x_n \in \Ss$, 
for every projection $p \in \K_{\Mul}$, we can find projections
$p_1, ..., p_m \in \K_{\Mul}$  with $p \leq p_k$ for $1 \leq k \leq m$ and
$\alpha_1, ..., \alpha_m \in [0,1]$ with $\sum_{k=1}^m \alpha_k = 1$ (i.e., coefficients for a 
convex combination) such that if we define $e =_{df} \sum_{k=1}^m \alpha_k p_k$, then  
$$\| x_j e - e x_j \| < \epsilon \makebox{ for all } 1 \leq j \leq n.$$
Note that $e p = p$, and the support projection of $e$
is also an element of $\K_{\Mul}$ (since $p_1 \vee p_2 \vee ... \vee p_m
\in \K_{\Mul}$). 
Also, for any given finite subset $\F \subseteq \Ss_K$, 
we can choose the projection $p \in \K_{\Mul}$ so that $\| py - y 
\| < \epsilon$ for all $y \in \F$.             

By an inductive construction using the statements in the previous
paragraph, we can construct a sequence $\{ e_n \}$, in $\K_{\Mul}$,
with the required properties.  
(Note also that since $\Mul$ has separable predual, we can find an increasing
sequence $\{ p'_l \}_{l=1}^{\infty}$ of projections in $\K_{\Mul}$
 such that $p'_l \rightarrow
1_{\Mul}$ in the weak* topology.  We can construct
 the sequence $\{ e_n \}$ so that for all $l$, $p'_l \leq e_n$ for sufficiently large $n$.) 
 
\end{proof}

The next result is a variation on results  from the not yet 
available \cite{GiordanoKaftalNg}.

\begin{lem} \label{lem:Reduction}
Let $\Mul$ be a semifinite von Neumann factor with separable predual.
Let
$\Ss \subset \Mul$ be a countable subset, and 
$\Ss_K \subset \K_{\Mul}$ be a countable subset.

Then we can find a separable simple stable C*-subalgebra 
$\B$ 
for which 
$$\Ss_K \subset \B \subseteq \K_{\Mul} \makebox{ and  } 
\Ss \subset \Mul(\B) \subseteq \Mul.$$\\ 
\noindent (Note that $\Mul(\B) \subset \Mul$ by Lemma \ref{lem:MultiplierAlgebraOfSubalgebraOfVNA}.)  
\end{lem}

\begin{proof}

Since $\Ss \subseteq \Mul$ is countable, let
$\Ss = \{ x_n : n \geq 1 \}$ be an enumeration of $\Ss$.

Plug $\Mul$, $\Ss$ and $\Ss_K$ 
into Lemma \ref{lem:QuasiCentralSOTApproximateUnit},  to get a sequence
$\{ e_n \}$, in $(\K_{\Mul})_+$, which satisfies the conclusions of Lemma
\ref{lem:QuasiCentralSOTApproximateUnit}.  We may assume that $e_0 =_{df} 0$. By replacing
$\{ e_n \}$ with a subsequence if necessary, we may assume that for all $m \geq n \geq 1$,
\begin{equation} \label{equ:StrongQuasicentralizeS}
\| (e_m - e_{m-1})^{1/2} x_n  - x_n (e_m - e_{m-1})^{1/2} \| 
< \frac{1}{10^m} \makebox{  and  }
\| e_m^{1/2} x_n - x_n e_m^{1/2} \| < \frac{1}{10^m}.
\end{equation}

Let $d =_{df} \sum_{n=1}^{\infty} \frac{e_n}{2^n} \in (\K_{\Mul})_+$, 
and let $\D \subset \K_{\Mul}$ be given by 
$\D =_{df} \overline{d \K_{\Mul} d}$.
Note that by Lemma \ref{lem:QuasiCentralSOTApproximateUnit} item (3), 
$\T_K \subset \D$.     
 
Next for each $n$, let $p_n \in \Mul$ be the support projection of
$e_n$;  since $e_{n+1} e_n = e_n$, $e_{n+1} p_n = p_n$, and hence,
$p_n \in \D$.  
Now by Lemma \ref{lem:QuasiCentralSOTApproximateUnit} item (4), 
$e_n \nearrow 1_{\Mul}$ in the weak* topology.  Hence, for each $n$, we 
can find $m_n \geq n + 3$ 
and a partial isometry $v_n \in \D$ with
\begin{equation} \label{equ:GetHJ1}     
v_n^* v_n = p_n \makebox{ and  } (e_{m_n} - e_{n+2}) v_n v_n^* = v_n v_n^*, 
\end{equation}
and hence,
\begin{equation} \label{equ:GetHJ2}
v_n e_n v_n^* \in Her(e_{m_n} - e_{n+2}). 
\end{equation}

Let $\B_0 \subseteq \K_{\Mul}$ be the separable C*-subalgebra that is given by
\begin{equation} \label{equ:B0Definition}
\B_0 =_{df} C^*(\Ss_K \cup \{ 
\makebox{ }e_m, \makebox{ } v_m, \makebox{ } e_m^{1/2} x e_m^{1/2}, \makebox{ } (e_m - e_{m-1})^{1/2} x
(e_m - e_{m-1})^{1/2}  :  x \in \Ss, 
\makebox{ } m \geq 1 \}).\end{equation}
From the above, we must have that $\B_0 \subset \D$. 
Note also that for all $m$, $p_m = v_m^* v_m \in \B_0$.

Now since $\D$ is simple, and since simplicity is a
separably inheritable property (see \cite{BlackadarOABook}
Definition II.8.5.1 and II.8.5.6), let $\B$ be a separable
simple C*-algebra 
such that 
$$\B_0 \subseteq \B \subset \D.$$ 
Note that since $\{ e_n \}$ is a sequential approximate unit for $\D$,
$\{ e_n \}$ is a sequential approximate unit for $\B$.
Hence, $\{ p_n \}$ is a sequential approximate for $\B$, consisting of
projections.

We next prove that $\B$ is stable.  Let $p \in \B$ be an arbitrary
projection.  Since $\B = \overline{\bigcup_{n=1}^{\infty} e_n \B e_n}$,
and since $e_{n+1} e_n = e_n$ for all $n$, we can choose $N \geq 1$ and
a projection $p' \in e_N \B e_N$ for which 
$\| p - p' \| < \frac{1}{10}$.  As a consequence, $p \sim p'$ in $\B$.
By (\ref{equ:GetHJ1}) and (\ref{equ:GetHJ2}), we can find a projection
$q' \in \B$ with $q' \perp p'$ such that
$q' \sim p' \sim p$. 
Now since $p \approx_{\frac{1}{10}} p'$, 
$\| q' p \| < \frac{1}{10}$.  Hence, we can find a projection $q \in \B$
with $q \perp p$ and $q \sim q' \sim p$.  Since $p$ is arbitrary, we have
proven that $\B$ satisfies the ``if" condition in the statement of
\cite{HjelmborgRordam} Theorem 3.3.  Hence, since $\B$ is a separable
C*-algebra with a sequential approximate unit consisting of
projections, by \cite{HjelmborgRordam} Theorem 3.3, 
$\B$ is stable.

We next prove that $\Ss \subseteq \Mul(\B)$. 
Let $x \in \Ss$ be arbitrary. Hence, since $\Ss = \{ x_n : n \geq 1 \}$,  
 choose $N \geq 1$ such that $x = x_N$.   
Let $\delta > 0$ be arbitrary.
Choose $M \geq N$ for which
$\frac{1}{10^{M-3}} < \delta$.
Hence, we have that
\begin{eqnarray}
\nonumber  x
& = & \left(e_M + \sum_{m=M}^{\infty} (e_{m+1} - e_m)\right) x \makebox{  (since $e_m \rightarrow 1_{\Mul}$
 in the weak* topology)}\\
\nonumber & \approx_{\delta} &
e_M^{1/2} x e_M^{1/2} + \sum_{m=M}^{\infty} (e_{m+1} - e_m)^{1/2} x (e_{m+1} - e_m)^{1/2}\\
\label{equ:Quasicentralizex} & & \makebox{  (by (\ref{equ:StrongQuasicentralizeS}) and
the definitions of } \delta, \makebox{ } M, 
\makebox{ and  }x\makebox{)}\\
\nonumber
\end{eqnarray}
By (\ref{equ:B0Definition}), since $\B_0 \subseteq \B$, and since $\{ e_n \}$ is an approximate
unit for $\B$, we have that 
$$e_M^{1/2} x e_M^{1/2} + \sum_{m=M}^{\infty} (e_{m+1} - e_m)^{1/2} x (e_{m+1} - e_m)^{1/2}
\in \Mul(\B),$$  
where the sum converges strictly in $\Mul(\B)$.
Hence, by (\ref{equ:Quasicentralizex}), $x$ is within a norm distance $\delta$ of an
element of $\Mul(\B) \subseteq \Mul$.
Since $\delta > 0$ was arbitrary, $x \in \Mul(\B)$.
Since $x \in \Ss$ was arbitrary,
$\Ss \subseteq \Mul(\B)$. 
  
Finally, by the definition of $\B_0$ (see 
(\ref{equ:B0Definition})),
$\Ss_K \subseteq \B_0  \subseteq \B$, and we are done. 
\end{proof}

Recall that for a C*-algebra $\D$, $S\D =_{df} C_0(0,1) \otimes \D$
is the \emph{suspension} of $\D$.

\begin{lem} \label{lem:ReductionJazz}
Let $\Mul$ be a semifinite von Neumann factor with separable predual,
and recall that $\K_{\Mul}$ denotes the Breuer ideal of $\Mul$.  
Let $\T, \T_K, \T_1, \T_{1,K}$ be countable sets where
$$\T \subset \Mul, \makebox{ } \T_K \subset \K_{\Mul}, 
\makebox{ } \T_1 \subset \Mul(S \K_{\Mul}), \makebox{ and }
\T_{1,K} \subset S \K_{\Mul}.$$

Then we can find a separable simple stable 
C*-algebra $\B \subseteq \K_{\Mul}$ 
such that  
$$\T \subset \Mul(\B), \makebox{ } \T_K \subset \B,
\makebox{ } \T_1 \subset \Mul(S \B), \makebox{ and }
\T_{1,K} \subset S \B.$$

\noindent (Note that since $S \B \subset S \K_{\Mul}$,
$\Mul(S \B) \subset (S \B)^{**} \subset (S \K_{\Mul})^{**}$.
Since also $\Mul(S \K_{\Mul}) \subset (S \K_{\Mul})^{**}$, the 
statement ``$\T_1 \subset \Mul(S \B)$" makes sense.)\\ 
\end{lem}

\begin{proof}
Let $E \subset (0,1)$ be a countable dense subset.

We may view each $f \in \T_1$ as a function in 
$C_b((0,1), \Mul)_{stri} = \Mul(S \K_{\Mul})$, 
where $C_b((0,1), \Mul)_{stri}$ denotes the
set of all bounded 
strictly continuous functions from $(0,1)$ to $\Mul(\K_{\Mul}) = \Mul$ (see
Lemma \ref{lem:MultiplierAlgebraOfSubalgebraOfVNA}).  

Apply Lemma \ref{lem:Reduction} to 
$$\T_2 =_{df} \T \cup \{ f(t) : f \in \T_1 \makebox{ and } t \in E \} \subset \Mul$$
and
$$\T_{2,K} =_{df} \T_K \cup \{ g(t) : g \in \T_{1,K} \makebox{ and } t \in E \} \subset \K_{\Mul}$$
to get a separable stable simple C*-algebra $\B$ 
such that
$$\T_{2, K} \subset \B \subset \K_{\Mul} \makebox{  and  } \T_2 \subset \Mul(\B)  \subset \Mul.$$

\noindent \emph{Claim 1:} For all $f \in \T_1$ and all $t \in (0,1)$, $f(t) \in \Mul(\B)$.\\
\noindent \emph{Proof of Claim 1:} 
Let $f \in \T_1$ and $t \in (0,1)$ be arbitrary.
Let $\{ t_n \}$ be a sequence in $E$ for which
$t_n \rightarrow t$.  Since $f \in C_b((0,1), \Mul)_{stri}$, 
$f(t_n) \rightarrow f(t)$ strictly in $\Mul(\K_{\Mul}) = \Mul$.    
Therefore, for all $b \in \B \subset \K_{\Mul}$, 
$f(t_n)b \rightarrow f(t)b$ and $b f(t_n)  \rightarrow b f(t)$ in the norm topology.
Since $\T_2 \subset \Mul(\B)$, $f(t_n) \in \Mul(\B)$, and hence,
for all $b \in \B$,  $f(t_n) b, b f(t_n) \in \B$, for all $n$. 
Hence, for all $b \in \B$, $f(t)b, b f(t) \in \B$.
Hence, $f(t) \in \Mul(\B)$.  Since $f$, $t$ are arbitrary, we have proven the Claim.\\
\noindent \emph{End of proof of Claim 1.}\\

By an argument similar to (and easier than) that of Claim 1, we can show that
for all $g \in \T_{1,K}$ 
and all $t \in (0,1)$, $g(t) \in \B$.  From this, it is not hard to see
that $\T_{1,K} \subset S \B$.  

We can finish the argument by proving the following Claim:\\ 

\noindent \emph{Claim 2:}  
$\T_1 \subset \Mul(S \B)$.\\
\noindent \emph{Proof of Claim 2:} Let $f \in T_1$ be arbitrary.
So $f \in C_b((0,1), \Mul)_{stri} = \Mul(S \K_{\Mul})$. Hence, for all $b \in
\B \subset \K_{\Mul}$ the maps $(0,1) \rightarrow \K_{\Mul}$ given by 
$t \mapsto f(t)b$ and $t \mapsto b f(t)$ are norm continuous. 
But by Claim 1, for all $t \in (0,1)$,
$f(t) \in \Mul(\B)$, and hence, for all $b \in \B$, $f(t)b, bf(t) \in \B$.  
Hence, $f \in C_b((0,1), \Mul(\B))_{stri} = \Mul(S \B)$.\\
\noindent{End of proof of Claim 2.}
\end{proof}

\subsection{Concrete and technical reduction arguments}

In this subsection, we provide three concrete, technical reduction arguments, which
are applications of the more abstract reduction results of the previous
part.  While these three lemmas and their proofs look technical, 
they are conceptually not so difficult variations
of each other.  So we will sketch 
the proofs for the first 2 lemmas, and leave the
third to the reader.

\begin{lem}
\label{lem:PAsympUEReduction}
Let $\A$ be a separable nuclear C*-algebra, and let $\Mul$ be a semifinite von Neumann factor 
with separable predual.  

Let $\phi, \psi : \A \rightarrow \Mul$ be two essential trivial extensions 
such that the following statements are true:
\begin{enumerate}
\item
 Either both $\phi$ and $\psi$ (and hence $\A$) are (strongly) unital or both
$\phi$ and $\psi$ have large complement.
\item 
$\phi(a) - \psi(a) \in \K_{\Mul}$ for all $a \in \A$. 
\end{enumerate}

Then there exists a separable simple stable C*-subalgebra 
$\B \subset \K_{\Mul}$ 
such that the following statements are  true:

\begin{enumerate}
\item[(a)]
$\{ 1_{\Mul} \} \cup ran(\phi) \cup ran(\psi) \subseteq \Mul(\B) \subseteq \Mul$.
\item[(b)] Either both $\phi$ and $\psi$ (and hence $\A$) are (strongly)
 unital or both have large 
complement (as maps into $\Mul(\B)$).   
\item[(c)]  
$\phi(a) - \psi(a) \in \B$ for
all $a \in \A$.
\item[(d)] For all $a \in \A_+ - \{ 0 \}$, there exist $x_{1,a}, y_{1,a} \in \C(\B)$ 
for which
$x_{1,a} (\pi \circ \phi)(a) (x_{1,a})^* = 1 = y_{1,a} (\pi \circ \psi)(a) (y_{1,a})^*$. 
\end{enumerate}
(Recall that $\Mul(\B) \subseteq \Mul$ by  
Lemma \ref{lem:MultiplierAlgebraOfSubalgebraOfVNA}.)    

Suppose, in addition, that
we have that
\begin{enumerate}
\item[(3)] either 
\begin{enumerate}
\item[i.] $\phi \sim_{pasymp} \psi$ (as maps 
into $\Mul$) or 
\item[ii.] $\phi \sim_{asymp, \K_{\Mul}} \psi$ (as maps into $\Mul$), 
where the path of unitaries 
is in $U(\pi_M^{-1}((\pi_M \circ \phi)(\A)'))_0$ 
\end{enumerate}
\end{enumerate}
Then we can choose $\B$, as above, so that additionally,
\begin{enumerate}
\item[(e)] either 
\begin{enumerate}
\item[i.] $\phi \sim_{pasymp} \psi$, as maps into $\Mul(\B)$,
or 
\item[ii.]  $\phi \sim_{asymp, \B} \psi$   
(as maps into $\Mul(\B)$), where the path of unitaries
is in $U(\pi_{B}^{-1}((\pi_B \circ \phi)(\A)'))_0$. 
\end{enumerate}
\end{enumerate}
\end{lem}

\begin{proof}
We will prove the case where $\phi$ and $\psi$ both have large complement
and $\phi \sim_{pasymp} \psi$ (as maps into $\Mul$).
 The proofs of the other cases are similar.

So suppose that, additionally (to the other hypotheses),
 $\phi$ and $\psi$ both have large complement, and   
$\phi \sim_{pasymp} \psi$ as maps into $\Mul(\K_{\Mul}) = \Mul$. So  
we can find a norm-continuous
path   
$\{ u_t \}_{t \in [1, \infty)}$  of unitaries in 
$\mathbb{C}1_{\Mul} + 
\K_{\Mul}$ such that
for all $a \in \A$, 
\begin{eqnarray}
\label{Jan120251AM} & & u_t \phi(a) u_t^* - \psi(a) \in \K_{\Mul} \makebox{ for all } t \in [1, \infty) \makebox{ and }\\
\label{Jan120252AM} & & \| u_t \phi(a) u_t^* - \psi(a) \| \rightarrow 0 
\makebox{  as  } t \rightarrow \infty.
\end{eqnarray}

Since $\phi$ and $\psi$ both have large complement, let $v_1, v_2 \in \Mul$ be partial 
isometries such that
$$v_1^* v_1 = 1_{\Mul} = v_2^* v_2,
\makebox{ }v_1 v_1^* \perp \phi(\A), \makebox{  and  } v_2 v_2^* \perp
 \psi(\A).$$

Note that since $\Mul/\K_{\Mul}$ is simple purely infinite, for all
$a \in \A_+ - \{ 0 \}$ with $\| a \| = 1$, we can find $x_a, y_a \in \Mul/\K_{\Mul}$
with $\| x_a \|, \| y_a \| < 2$ for
which 
$$x_a (\pi \circ \phi)(a) x_a^* = 1 = y_a (\pi \circ \psi)(a) y_a^*.$$  
  
Hence, for all $a \in \A_+ - \{ 0 \}$ with $\| a \| = 1$, lift $x_{a}, y_{a}$
to $x'_{a}, y'_{a} \in \Mul$ respectively (so $\pi(x'_{a}) = x_{a}$ and
$\pi(y'_{a}) = y_{a}$) with $\| x'_a \|, \| y'_a \| < 2$.  Then for all $a \in \A_+ - \{ 0 \}$
with $\| a \| = 1$,  
we can find $k_{1,a}, k_{2, a} \in  \K_{\Mul}$ for which 
\begin{equation} \label{equ:Jan1202513AM}
x'_a \phi(a) (x'_a)^* = 1_{\Mul} + k_{1,a} 
\makebox{  and  }
y'_a \psi(a) (y'_a)^* = 1_{\Mul} + k_{2,a}.  
\end{equation}

Let $\{ a_n : n \geq 1\}$ be a countable dense subset of the closed unit sphere of $\A_+$, and 
let $\{ u_n : n \geq 1 \}$ be an enumeration of the  countable  set 
$\{ u_t : t \in \mathbb{Q} \cap
[1, \infty) \}$.
Also, for all $n \geq 1$, $u_n$ has the form 
\begin{equation} \label{equ:Jan120253AM} u_n = \alpha_n 1_{\Mul} + b_n \makebox{ where  } \alpha_n \in S^1 \subset \mathbb{C} \makebox{  and  } 
b_n \in \K_{\Mul}. \end{equation}

Let $$\Ss =_{df} \{ 1_{\Mul}, v_1, v_2 \} \cup  
\{ x'_{a_n}, y'_{a_n}, \phi(a_n), \psi(a_n),  u_n : n \geq 1 \} \subseteq \Mul.$$     
Let $$\Ss_K =_{df}  
\{ k_{1, a_n}, k_{2, a_n}, b_n, \phi(a_n) - \psi(a_n), \makebox{ }
u_m \phi(a_n) u_m^* - \psi(a_n) : n, m \geq 1 \} \subseteq \K_{\Mul}.$$
Plug $\Mul$, $\Ss$ and $\Ss_K$ into Lemma \ref{lem:Reduction},  
to get a separable simple stable C*-algebra $\B$ 
such that          
$\Ss \subset \Mul(\B) \subset \Mul$, and
$\Ss_K \subset \B \subset \K_{\Mul}$.
(Recall that $\Mul(\B) \subseteq \Mul$ by Lemma \ref{lem:MultiplierAlgebraOfSubalgebraOfVNA}.)

Since  
$\{ a_n : n \geq 1\}$ is a countable dense subset of the closed unit sphere of $\A_+$,
and $\mathbb{Q} \cap [1, \infty)$ is dense in $[1, \infty)$, and
by the definitions of $\Ss$ and $\Ss_K$, we get the conclusions (a)-(e).\\ 
\end{proof}

\begin{lem}
\label{lem:ExtReduction}
Let $\A$ be a separable nuclear C*-algebra, and let
$\Mul$ be a semifinite von Neumann factor with separable predual.

Let $\phi, \psi: \A  \rightarrow \Mul$ be essential trivial extensions, and
let $\theta: \A \rightarrow \Mul(S \K_{\Mul})$ be a trivial extension such that
the following statements are true:
\begin{enumerate}
\item Either $\phi, \psi$ and $\theta$ are all unital (and hence $\A$ is unital)
or $\phi, \psi$ and $\theta$ all have large complement.
\item $\phi(a) - \psi(a) \in \K_{\Mul}$ for all $a \in \A$.
\item There exists a norm-continuous path $\{ u_s \}_{s \in [1, \infty)}$ 
of unitaries in $\Mul(S \K_{\Mul})/S\K_{\Mul}$ such that for all $a \in \A$, 
$u_s (\pi \circ \theta)(a) u_s^* \rightarrow \pi(\{ (1-t) \phi(a) + t \psi(a) \}_{t \in (0,
1)})$
in norm.
\end{enumerate}

Then there exists a separable simple stable C*-subalgebra $\B \subset \K_{\Mul}$
such that the following statements are true:
\begin{enumerate}
\item[(a)] $\{ 1_{\Mul} \} \cup ran(\phi) \cup ran(\psi) 
\subset \Mul(\B) \subset \Mul$ and 
$\{ 1_{\Mul(S \K_{\Mul})} \} \cup ran(\theta) \subset \Mul(S \B)$. 
\item[(b)] As maps into $\Mul(\B)$ (or $\Mul(S \B)$),
either $\phi, \psi$ (resp. $\theta$) are all unital (and hence $\A$ is unital)
or $\phi, \psi$ (resp. $\theta$) all have large complement.
\item[(c)] $\phi(a) - \psi(a) \in \B$ for all $a \in \A$.
\item[(d)] For all $a \in \A_+ - \{ 0 \}$, there exist $x_{1,a}, y_{1,a} \in \C(\B)$ such that
$x_{1,a} (\pi \circ \phi)(a) x_{1,a}^* = 1 = 
y_{1,a} (\pi \circ \psi)(a) y_{1,a}^*$.
\item[(e)] There exists a norm-continuous path $\{ w_s \}_{s \in [1, \infty)}$
of unitaries in $\C(S\B)$ such that for all $a \in \A$, 
$w_s (\pi \circ \theta)(a) w_s^* \rightarrow \pi(\{ (1-t) \phi(a) + t \psi(a) \}_{t
\in (0,1)})$
in norm (as maps into $\C(S\B)$).
\end{enumerate}
\end{lem}

\begin{proof}
Let us prove the case where $\phi, \psi, \theta$ all have large
complement.  The proof, for the unital case, is similar (with minor
changes).

Since $\phi$ and $\psi$ have large complement, let $v_1, v_2 \in 
\Mul$ be partial isometries such that 
$$v_1^* v_1 = 1_{\Mul} = v_2^* v_2, v_1 v_1^* \perp \phi(\A), 
\makebox{  and  } v_2 v_2^* \perp \psi(\A).$$
Also, since $\theta$ has large complement, 
let $v_3 \in \Mul(S \K_{\Mul})$ be a partial isometry so that
$$v_3^* v_3 = 1_{\Mul(S \K_{\Mul})} \makebox{  and  }
v_3 v_3^* \perp \theta(\A).$$ 

Since $\C(\K_{\Mul}) = \Mul/\K_{\Mul}$ is simple purely infinite, 
for each $a \in \A_+$ with $\| a \| = 1$,  
we can find $x_a, y_a \in \Mul/\K_{\Mul}$ with $\| x_a \|, \| y_a \| < 2$
such that
$$x_a (\pi \circ \phi)(a) x_a^* = 1_{\Mul} = 
y_a (\pi \circ \psi)(a) y_a^*.$$
So for all $a \in \A_+$ with $\| a \| = 1$, lift $x_a, y_a$ to
$x'_a, y'_a \in \Mul$ (so $\pi(x'_a) = x_a$ and $\pi(y'_a) = y_a$)
such that 
$\| x'_a \|, \| y'_a \| < 2$.  So for all $a \in \A_+$ with $\| a \| = 1$, 
let $k_{1,a}, k_{2,a} \in \K_{\Mul}$ be such that  
$$x'_a \phi(a) (x'_a)^*  = 1 + k_{1,a}
\makebox{  and  } y'_a \psi(a) (y'_a)^* = 1 + k_{2,a}.$$

Viewing $\{ u_s \}_{s \in [1, \infty)}$ as a unitary element of
$C_b([1, \infty)) \otimes \left(\Mul(S \K_{\Mul})/S \K_{\Mul})\right)$, 
we can lift $\{ u_s \}_{s \in [1, \infty)}$ to a contractive 
element of   
$C_b([1, \infty)) \otimes \Mul(S \K_{\Mul})$, which we can view as a 
norm-continuous path $\{ \widetilde{u}_s \}_{s \in [1, \infty)}$ 
of contractions in $\Mul(S \K_{\Mul})$ (so $\widetilde{u}_s
\in \Mul(S \K_{\Mul})$ for all $s \in [1, \infty)$).  
For all $s \in [1, \infty)$, 
let $c_{1,s}, c_{2,s} \in S \K_{\Mul}$ be such that 
$$\widetilde{u}_s^* \widetilde{u}_s = 1_{\Mul(S \K_{\Mul})} + c_{1,s} \makebox{  and  }
\widetilde{u}_s \widetilde{u}_s^* = 1_{\Mul(S \K_{\Mul})} + c_{2,s}.$$

Let $\{ a_n \}_{n=1}^{\infty}$ be a dense sequence in the unit sphere
of $\A_+$.  
Let $E =_{df} \{ s_l : 1 \leq l < \infty \}$ be a countable dense subset
of $[1, \infty)$.  For each $n, l \geq 1$, let $b_{n,l} \in 
S \K_{\Mul}$ be such that, in $\Mul(S \K_{\Mul})$, 
the (norm-) distance between 
$\widetilde{u}_{s_l} \theta(a_n) \widetilde{u}_{s_l}^* + b_{n,l}$ and 
$\{ (1-t) \phi(a_n) + t \psi(a_n) \}_{t \in (0,1)}$
is at most                             
$$\| u_{s_l} (\pi \circ \theta)(a_n) u_{s_l}^* - 
\pi(\{(1-t) \phi(a_n) + t \psi(a_n) \}_{t \in (0,1)}) \| + 
\frac{1}{n + l},$$
(where the last norm is for an element in $\Mul(S \K_{\Mul})/S\K_{\Mul}$). 

Now let 
$$\T =_{df} \{ 1_{\Mul}, \phi(a_n), \psi(a_n), v_1, v_2, x'_{a_n}, 
y'_{a_n} : n \geq 1 \},$$
$$\T_K =_{df} \{ k_{1, a_n}, k_{2, a_n}, \phi(a_n) - \psi(a_n) : n \geq 1 \},$$
$$\T_1 =_{df} \{ 1_{\Mul(S \K_{\Mul})}, \theta(a_n), v_3, \widetilde{u}_{s_l} 
: n, l \geq 1 \},$$
and
$$\T_{1,K} =_{df} \{ c_{1, s_l}, c_{2, s_l}, b_{n,l} : n, l \geq 1 \}.$$

Plug $\T, \T_K, \T_1$ and $\T_{1,K}$ into 
Lemma \ref{lem:ReductionJazz} to get a separable simple stable C*-algebra
$\B$. It is not hard to see that we get
conclusions (a) to (d). In fact, it is also not hard to see that
we get conclusion (e), but let us elaborate a bit on this.  For each $l \geq 1$, let $w_{s_l} =_{df}
 \pi(\widetilde{u}_{s_l})$ which, by our construction,
 is a unitary in $\C(S\B)$.
For each $s \in [1, \infty)$, choose a sequence $\{ s'(j) \}$ in $E$
such that $s'(j) \rightarrow s$.  Then by our construction
$\{ w_{s'(j)} \}$ is a Cauchy sequence of unitaries
 in $\C(S\B)$, and thus, we can
find a unitary $w_s \in \C(S\B)$ such that $w_{s'(j)} \rightarrow w_s$ in norm.
Moreover, by our construction,  
$\{ w_s \}_{s \in [1, \infty)}$ is a norm continuous path of unitaries
in $\C(\B)$, which together with the above, gives the conclusion of  
the Lemma.        
\end{proof}

\begin{lem}
\label{lem:VoiculescuWvNReduction}
Let $\A$ be a separable nuclear C*-algebra, and let 
$\Mul$ be a semifinite von Neumann factor with separable predual.

Let $\phi : \A \rightarrow \Mul/\K_{\Mul}$ be an essential extension,
and $\psi_0 : \A \rightarrow \Mul$ an essential trivial extension such 
that 
either $\phi$ is unital and $\psi_0$ is (strongly) unital, or 
$\phi$ has large complement.

Since $\A$ is nuclear, by \cite{ChoiEffros}, 
let $\phi_0 : \A \rightarrow \Mul$ be a completely positive contractive lift
of $\phi$.

Then there exists a separable simple stable C*-subalgebra 
$\B \subset \K_{\Mul}$ such the following statements are true:
\begin{enumerate}
\item[(a)] $\{ 1_{\Mul}\} \cup ran(\phi_0) \cup ran(\psi_0) \subseteq \Mul(\B) \subset \Mul$.
\item[(b)] $\pi \circ \phi_0 : \A \rightarrow \C(\B)$ is an essential extension
(in particular, it is a *-homomorphism).
\item[(c)] As maps into $\C(\B)$ (and $\Mul(\B)$), 
either $\pi \circ \phi_0$ is unital (resp. and $\psi_0$ is (strongly) unital),
or $\pi \circ \phi_0$ has large complement.
\item[(d)] For all $a \in \A_+ - \{ 0 \}$, there exists an $x_a \in \C(\B)$ such that
$x_a (\pi \circ \phi_0)(a) x_a^* = 1$.     
\end{enumerate}
\end{lem}

\begin{proof}
The proof is a variation on the reduction arguments of  
Lemmas \ref{lem:PAsympUEReduction} and \ref{lem:ExtReduction}, and we leave this
to the reader.  
\end{proof}

\section{A Voiculescu--Weyl--von Neumann theorem, 
and the Paschke dual in the semifinite factor case}

\label{sect:AWvN}

Next, we recall some more notions from extension theory. While reading
what follows, the reader          
should refer back to the preliminary conventions and notation 
in \ref{para:Extensions1}.
\subsubsection{} \label{para:Extensions2}                           
Let $\A$ be a C*-algebra and suppose that either $\B$ is a separable stable
C*-algebra or $\B$ is the Breuer ideal of a semifinite von Neumann
factor $\Mul$ (i.e., $\B  = \K_{\Mul}$) with separable predual. 
Let $\phi, \psi : \A \rightarrow \C(\B)$ 
be two extensions.  We say that $\phi$ and $\psi$
are \emph{unitarily equivalent} (and write $\phi \sim \psi$)
 if there exists a unitary 
$w \in \Mul(\B)$ such that 
\begin{equation} \label{equdf:UE}
\pi(w) \phi(a) \pi(w)^* = \psi(a) \makebox{  for all } a \in \A.
\end{equation}
If $\phi$ and $\psi$ are trivial extensions and $\phi_0, \psi_0 : \A
\rightarrow \Mul(\B)$ are *-homomorphisms for which 
$\phi = \pi \circ \phi_0$ and $\psi = \pi \circ \psi_0$, sometimes we 
write $\phi_0 \sim \psi_0$ to mean $\phi \sim \psi$.

Also the \emph{BDF sum} of $\phi$ and $\psi$ is defined to be
\begin{equation} 
\label{equdf:BDFSum}
(\phi \oplus \psi)(\cdot) =_{df}  \pi(S) \phi(\cdot) \pi(S)^* + 
\pi(T) \psi(\cdot) \pi(T)^*,
\end{equation}
\noindent where $S, T \in \Mul(\B)$ are isometries with
$SS^* + TT^* = 1$.    
Such $S$ and $T$ always exist since either $\B$ is  stable or 
$\B = \K_{\Mul}$.
The BDF sum is well-defined (independent of the choice of
$S, T$) up to unitary equivalence.  
Finally, if $\phi, \psi: \A \rightarrow \C(\B)$ are trivial extensions,
and if $\phi_0, \psi_0 : \A \rightarrow \Mul(\B)$ are *-homomorphisms which
lift $\phi$, $\psi$ respectively (i.e., $\phi = \pi \circ \phi_0$ and 
$\psi = \pi \circ \psi_0$), then we sometimes write $\phi_0 \oplus \psi_0$ 
to mean the BDF sum $\phi \oplus \psi$.

Suppose that $\phi : \A \rightarrow \C(\B)$ is a essential 
extension 
with large complement (or which is unital).
Then $\phi$ is said to be \emph{absorbing} if for every 
essential trivial extension $\sigma : \A \rightarrow \C(\B)$ (resp. which
is strongly unital),
the BDF sum $\phi \oplus \sigma \sim \phi$.\\

The next computation is standard, but we present it for the convenience
of the reader.

\begin{lem} \label{lem:Feb220251AM}
Let $\B$ be a separable stable C*-algebra.
Suppose that $a \in \C(\B)_+$ and $x \in \C(\B)$ for which
$x a x^* = 1_{\C(\B)}$. 

Then for any $A \in \Mul(\B)_+$ for which $\pi(A) = a$, there exists
an $X \in \Mul(\B)$ such that $X A X^* = 1_{\Mul(\B)}$.
\end{lem}

\begin{proof}
 Lift $x$ to $Y \in \Mul(\B)$; i.e., $\pi(Y) = x$.
So we can find $b \in \B_{SA}$ where  $Y A Y^* = 1 + b$.
Since $\B$ is stable, we can find an isometry $S \in \Mul(\B)$ for which
$\| b S \| < \frac{1}{10}$.
Hence,
$S^* Y A Y^* S = S^* S  + S^* b S = 1 + S^* b S$.
Hence, $S^* Y A Y^* S \approx_{\frac{1}{10}} 1$.
Hence, $S^* Y A Y^* S$ is a positive invertible element of $\Mul(\B)$.
Hence, we can find $Z \in \Mul(\B)$ where
$Z S^* Y A Y^* S Z^* = 1$.
\end{proof}

The next result follows from \cite{KucerovskyFredholmTriples}
Lemma 1.3 and Theorem 1.4 (see also
\cite{GiordanoKaftalNg}), but we give a (different)
short proof for the convenience of
the reader.

\begin{lem} \label{lem:EKAbsorb}
Let $\A$ be a separable nuclear C*-algebra. Let $\B$ be a separable stable C*-algebra.
Suppose that $\phi : \A \rightarrow \C(\B)$ is an
 essential extension
such that either $\phi$ is unital or  $\phi$ has
large complement.

Then the following statements are equivalent:
\begin{enumerate}
\item
For all $a \in \A_+ - \{ 0 \}$, there exists an $x \in \C(\B)$ for which
\begin{equation} x \phi(a) x^* = 1_{\C(\B)}.  
\end{equation}
\item
$\phi$ is an absorbing  extension (for both the unital and large
complement cases; see \ref{para:Extensions2}).
\end{enumerate}
\end{lem}

\begin{proof}

(1) $\Rightarrow$ (2).
Let $\E \subseteq \Mul(\B)$ be the C*-subalgebra given by
$\E =_{df} \pi^{-1}(\phi(\A))$. 
Then we get an essential extension
$$0 \rightarrow \B \rightarrow \E \rightarrow \A \rightarrow 0$$
whose Busby invariant is $\phi$.

We now check the Elliott--Kucerovsky purely large condition (see
\cite{ElliottKucerovsky} Definition 1).
Let $c \in \E - \B$.
We want to show that $\overline{c \B c^*}$ contains a C*-subalgebra
which is stable and full in $\B$.  Since
$\overline{ c \B c^*} = \overline{cc^* \B cc^*}$, we may assume that
$c \geq 0$.  Now, $\pi(c) = \phi(a)$, where
$a \in \A_+  - \{ 0 \}$.
By hypothesis, there exists an $x \in \C(\B)$ for which
$x (\pi \circ \phi)(a) x^* = 1_{\C(\B)}$.   Hence, by
Lemma \ref{lem:Feb220251AM},  we can find an $X \in \Mul(\B)$
such that $Xc X^* = 1_{\Mul(\B)}$. 
Hence, for all $b \in \B_+$, if we define $z =_{df} b^{1/2} X c^{1/2}
\in \B$,
then 
$z z^* = b^{1/2} X c X^* b^{1/2} = b$ and $z^* z = c^{1/2} X^* b X c^{1/2} \in
\overline{c \B c}$; and
thus, $b = zz^* \in Ideal(\overline{c \B c})$. Hence, $\overline{c \B c}$ cannot be contained in a
proper ideal of $\B$; i.e., $\overline{c \B c}$ is full in $\B$.

Again, since $X c X^* = 1_{\Mul(\B)}$,
we can find a projection $P \in \overline{c \Mul(\B) c}$ such that
$P \sim 1_{\Mul(\B)}$.  Hence, we can find a sequence of pairwise orthogonal
projections $\{ P_n \}$, in $\overline{c\Mul(\B)c}$,
for which $P_n \sim 1_{\Mul(\B)}$
for all $n$, and $P = \sum_{n=1}^{\infty} P_n$ where the sum converges
strictly in $\Mul(\B)$.

We now show that $\overline{c \B c}$ satisfies the Hjelmborg--Rordam
characterization of stability
(see \cite{HjelmborgRordam} Proposition 2.2 and Theorem 2.1).
Let $a' \in \overline{(c \B c)}_+$ be arbitrary. For simplicity, let us assume
that $\| a' \| \leq 1$.   Let $\epsilon > 0$ be
given.
For all $M \geq 1$, let $P_{s,M} \in \Mul(\B)$ be the partial sum
$P_{s,M} =_{df} \sum_{n=1}^M P_n$.
Since $P = \sum_{n=1}^{\infty} P_n$, where the sum converges strictly in
$\Mul(\B)$, we have that $P_{s,M} \rightarrow P$ strictly in $\Mul(\B)$ as
$M \rightarrow \infty$.   
Hence, $P_{s,M} + (1 - P) \rightarrow 1_{\Mul(\B)}$ strictly, as
$M \rightarrow \infty$.
Hence,
$$a_M =_{df} (P_{s,M} + (1 - P)) a' (P_{s,M} + (1-P)) \rightarrow a'
\makebox{ in norm, as } M \rightarrow \infty.$$
Note that $\forall  M \geq 1$, $a_M \in \overline{c \B c}$.
(E.g., $(1-P)a'(1-P) 
\in \overline{c\B c}$, since $a'\in \overline{c \B c}$ and
$P \in \overline{c \Mul(\B) c}$.)

Hence, choose $N \geq 1$ such that
$$\| a_N - a' \| < \epsilon.$$
Since $P_{N+1} \perp a_N$ and $P_{N+1} \sim 1_{\Mul(\B)}$, we can
find $d \in P_{N+1} \B P_{N+1} \subset \overline{c \B c}$
and $y \in \overline{a_N \B P_{N+1}}
\subseteq \overline{c \B c}$ such that
$$y^* y  =  d \makebox{ and } y y^* = a_N.$$
Since $a'$ and $\epsilon$ are arbitrary, by \cite{HjelmborgRordam}
Proposition 2.2 and Theorem 2.1, $\overline{c \B c}$ is stable.
Since $c \in \E_+ - \B$ is arbitrary,
$\phi$ satisfies the Elliott--Kucerovsky purely large condition
(\cite{ElliottKucerovsky} Definition 1). Hence, since $\A$ is nuclear,
by \cite{ElliottKucerovsky}  Theorem 6 and Corollary 16, $\phi$  is
absorbing.\\   

(2) $\Rightarrow$ (1).  Suppose that $\phi$ is absorbing.
Since $\B$ is stable, $\B \cong \B \otimes \K$, and we will work with
$\B \otimes \K$.
Let $\rho : \A \rightarrow 1_{\Mul(B)} \otimes \mathbb{B}(l_2) \subset
\Mul(\B) \otimes \Mul(\K) \subset
\Mul(\B \otimes \K)$ be an essential
trivial extension (which is (strongly) unital when $\phi$ is unital).
By the properties of $\mathbb{B}(l_2)$,
we have that for all $a \in \A_+ - \{ 0 \}$, there exists an
$X_a \in 1_{\Mul(\B)} \otimes
\mathbb{B}(l_2)$ for which $X_a \rho(a) X_a^* = 1_{1 \otimes \mathbb{B}(l_2)}$.
Let us now view $\rho$ as being a map into $\Mul(\B \otimes \K)$.
Hence, for all $a \in \A_+ - \{ 0 \}$, there exists an $x  \in
\C(\B \otimes \K)$, for which
$x (\pi \circ \rho)(a) x^* = 1$.
Hence, if we take the BDF sum $\phi \oplus (\pi \circ \rho)$,
then for all $a \in \A_+ - \{ 0 \}$,
there exists a $y  \in
\C(\B \otimes \K)$, for which
$y (\phi \oplus (\pi \circ \rho))(a)) y^* = 1$.
But since $\phi$ is absorbing (in the unital or nonunital sense, depending
on whether $\phi$ is unital or has large complement),
$\phi \sim \phi \oplus (\pi \circ \rho)$.
Hence, for all  $a \in \A_+ - \{ 0 \}$,
there exists a $z  \in
\C(\B \otimes \K)$, for which
$z \phi(a) z^* = 1$.
\end{proof}

We next prove an absorption result which is 
a generalization, to more general semifinite factors, of Voiculescu's 
``all essential extensions are absorbing" result for $\mathbb{B}(l_2)$
(see \cite{Voiculescu}). This result   
is also already present (often implicitly) 
in previous works (see, for example, 
\cite{LiShenShi}; see also  
\cite{HadwinMaShen}, \cite{GiordanoNg},
\cite{CiupercaGiordanoNgNiu},  \cite{GiordanoKaftalNg}).
We provide the explicit statement and short proof for the convenience of
the reader.

\begin{thm} \label{thm:VNAAbsorption}
Let $\A$ be a separable nuclear C*-algebra and $\Mul$ a semifinite factor
with separable predual.

Let $\phi : \A \rightarrow \Mul/\K_{\Mul}$ be an essential extension such
that either $\phi$ is unital or $\phi$ has large complement.

Then $\phi$ is absorbing (in either the unital or nonunital sense, depending
on whether $\phi$ is unital or has large complement). 
\end{thm}

\begin{proof}
Let us prove the result for the case where $\phi$ has large complement. 
The proof, for the case where $\phi$ is unital, is similar.  

So suppose that $\phi : \A \rightarrow \Mul/\K_{\Mul}$ is an essential
extension with large complement. 
Let $\psi_0 : \A \rightarrow \Mul$ be an arbitrary trivial extension.

Since $\A$ is nuclear, let $\phi_0 : \A \rightarrow \Mul$ be a completely
positive contractive lift of $\phi$.  

By Lemma \ref{lem:VoiculescuWvNReduction}, we can find a separable stable
simple C*-subalgebra $\B \subset \K_{\Mul}$ which satisfies the conclusions
of Lemma \ref{lem:VoiculescuWvNReduction} (where $\pi \circ \phi_0$ has
large complement).   By item (d) of Lemma \ref{lem:VoiculescuWvNReduction},
and by Lemma \ref{lem:EKAbsorb}, we can find a unitary $U \in \Mul(\B)$ such  
that for all $a \in \A$,
$$U(\phi_0(a) \oplus \psi_0(a))U^* - \phi_0(a) \in \B$$
(where all maps are taken to have codomain $\Mul(\B)$).
Since $\B \subset \K_{\Mul}$ and $\Mul(\B) \subset \Mul$ 
(unital C*-subalgebra), we have
that $U \in \Mul$ and for all $a \in \A$,
$$U(\phi_0(a) \oplus \psi_0(a))U^* - \phi_0(a) \in \K_{\Mul}$$
(where all maps are taken to have codomain $\Mul$).
Hence, as maps into $\Mul/\K_{\Mul}$, $\phi$ absorbs $\pi \circ \psi_0$.
\end{proof}

Towards proving the $K_1$ injectivity of the Paschke dual algebra
(see Theorem \ref{thm:PaschkeDualWellDefined}),
we need the next lemma which generalizes a result from \cite{LoreauxNg}.
     
Recall that if $\A \subseteq \C$ is an inclusion of C*-algebras,
$\A'$ is the commutant of $\A$ in $\C$; i.e., 
$\A' =_{df} \{ c \in \C : ca = ac, \makebox{ } \forall a \in \A \}$.

\begin{lem} \label{VNAlem2.3ofIEOT} 
 Let $\C$ be a unital C*-algebra and $\A\subseteq\C$
a separable nuclear unital
C*-subalgebra.
 Say that $u \in \A'(\subseteq\C)$ is a unitary.
Let 
$\Mul$ be a semifinite von Neumann factor with separable predual.
 Let $\phi: C^*(\A, u) \rightarrow \Mul$ be a (strongly)
unital trivial essential  
extension.

Then there exists a norm-continuous path of
unitaries $\{v_t\}_{t\in [0,1]}$ in
$(\pi\circ \phi(\A))' (\subseteq \Mul/\K_{\Mul})$  such that
$v_0 = \pi\circ\phi(u)$ and $v_1=1.$
\end{lem}

\begin{proof}
Firstly, note that since $\A$ is nuclear, $C(S^1) \otimes \A$ is nuclear.
Hence, since $C^*(\A, u)$ is a quotient of $C(S^1) \otimes \A$, 
$C^*(\A, u)$ is a nuclear C*-algebra.

Hence, by Lemma \ref{lem:PAsympUEReduction} (taking $\phi = \psi$) 
and Lemma \ref{lem:EKAbsorb},
we can find a separable simple
 stable C*-subalgebra $\B \subseteq \K_{\Mul}$
such that 
$\phi(C^*(\A,u)) \subseteq \Mul(\B)$ 
and $\phi : C^*(\A, u) \rightarrow \Mul(\B)$ is a (strongly) unital absorbing
trivial extension.  (Recall that by Lemma 
\ref{lem:MultiplierAlgebraOfSubalgebraOfVNA},
$\Mul(\B) \subseteq \Mul$.)  

Hence, by \cite{LoreauxNg} Lemma 2.3,   
we can find a  
norm-continuous path $\{ v'_t \}_{t \in [0,1]}$, of unitaries in 
$(\pi_{\B} \circ \phi(\A))' \subseteq \Mul(\B)/\B$, such that $v'_0 = \pi_{\B} \circ \phi(u)$ and $v'_1 = 1$.
Since $\B \subseteq \K_{\Mul}$ and since $\Mul(\B) \subseteq \Mul$ (unital
C*-subalgebra), we have a (not necessarily injective) unital *-homomorphism
$\Mul(\B)/\B \rightarrow \Mul/\K_{\Mul}$. 
And the image of $\{ v'_t \}_{t 
\in [0,1]}$ under this unital *-homomorphism is a norm-continuous path
$\{ v_t \}_{t \in [0,1]}$, of unitaries in
$(\pi_{{\Mul}} \circ \phi(\A))' \subseteq \Mul/\K_{\Mul}$, 
such that $v_0 = \pi_{{\Mul}} \circ \phi(u)$ 
and $v_1 = 1$.     
\end{proof}

Recall that a unital C*-algebra $\C$ is said to be    
\emph{$K_1$-injective} if for all $n \geq 1$, the usual map
$U(\M_n \otimes \C)/U(\M_n \otimes 
\C)_0 \rightarrow K_1(\C)$ is injective. 
The next theorem 
(Theorem \ref{thm:PaschkeDualWellDefined}) 
states that the Paschke
dual algebra  $(\pi_{\Mul} \circ \phi)(\A)'$ (where $\Mul$ is a
semifinite factor)
is properly infinite and $K_1$-injective.  This
answers, for a special case, a conjecture of 
Blanchard--Rohde--Rordam
which asks whether every unital properly infinite C*-algebra is
$K_1$-injective (\cite{BlanchardRohdeRordam}).

\begin{thm} \label{thm:PaschkeDualWellDefined}                                  
Let $\A$ be a unital separable nuclear C*-algebra, and let 
$\Mul$ be a semifinite von Neumann
factor with separable predual.

Let $\phi : \A \rightarrow \Mul$ be a (strongly) unital 
essential trivial extension.

Recall that 
$\pi$ denotes the canonical quotient map 
$\pi: \Mul\rightarrow \Mul/\K_{\Mul}$,  
and recall that 
$$(\pi \circ \phi)(\A)' =_{df} \{ x \in \Mul/\K_{\Mul} : 
x (\pi \circ\phi)(a) = 
(\pi \circ \phi)(a) x, \makebox{ }\forall a \in \A  \}.$$

Then the following statements hold:
\begin{enumerate}
\item $(\pi \circ \phi)(\A)' $ is C*-algebra and is 
(up to *-isomorphism)
independent of the choice of the unital essential trivial extension $\phi$.
\item  $(\pi \circ \phi)(\A)' $ is properly infinite.  In fact,
$(\pi \circ \phi(\A))'$ contains a unital copy of the Cuntz algebra
$O_2$.
\item $(\pi \circ \phi)(\A)' $ is $K_1$-injective. 
\end{enumerate}
\end{thm}

\begin{proof}
 
(1) follows from that  
if $\phi_0, \psi_1: \A \rightarrow \Mul$ are unital trivial essential extensions
then by Theorem \ref{thm:VNAAbsorption}, $\phi_0 \sim \phi_0 \oplus \phi_1
\sim \phi_1$. 

The argument for (2) is exactly the same as the argument for
\cite{LoreauxNg} Lemma 2.2(a), except that \cite{ElliottKucerovsky} is
replaced with (the present paper) Theorem \ref{thm:VNAAbsorption}.  

The argument for (3) exactly the same as the argument for \cite{LoreauxNg}
Lemma 2.4 and Theorem 2.5, except that \cite{ElliottKucerovsky} is 
replaced with (the present paper) Theorem \ref{thm:VNAAbsorption},
and \cite{LoreauxNg} 
Lemma 2.3 is replaced with (the present paper) 
Lemma \ref{VNAlem2.3ofIEOT}. Note that $\Mul/\K_{\Mul}$ is simple
purely infinite. 

\end{proof}
The Paschke dual algebra and its importance for uniqueness theorems are studied
in \cite{LoreauxNg} and \cite{LoreauxNgSutradhar2}.

\section{Main result}

\label{sect:MainResult}

Recall that for a nonunital C*-algebra $\D$, we let $\widetilde{\D}$ denote the unitization of $\D$.
For a general C*-algebra $\D$, we let
\[
\D^+ =_{df} \begin{cases}
\widetilde{\D} & \makebox{  if } \D \makebox{ is nonunital, and}\\
\D \oplus \mathbb{C} & \makebox{  if  } \D \makebox{  is unital.}
\end{cases}
\]

\begin{lem} \label{lem:LinDE}
Let $\Mul$ be a semifinite factor with separable predual, and let 
$\A$ be a separable nuclear C*-algebra. Suppose that
$\phi, \psi : \A \rightarrow \Mul$ are essential trivial extensions
such that $\phi(a) - \psi(a) \in \K_{\Mul}$ for all $a \in \A$;
and suppose that either both $\phi$ and $\psi$ are unital or both have
large complement. 

Suppose that there exists a separable, stable C*-subalgebra $\B \subset   
\K_{\Mul}$ such that 
\begin{enumerate}
\item $\{ 1_{\Mul} \} \cup \phi(\A) \cup 
\psi(\A) \subset \Mul(\B) \subset \Mul$, 
\item $\phi(a) - \psi(a) \in \B$ for all $a \in \A$,
\item $\phi$ and $\psi$, as maps into $\Mul(\B)$, either both are unital
or both have large complement,
\item $\phi$ and $\psi$, as maps into $\Mul(\B)$, are absorbing (either
in the unital or nonunital sense), and 
\item $[\phi, \psi] = 0$ in $KK(\A, \B)$. 
\end{enumerate}

Then there exists a norm-continuous path $\{ u_t \}_{t \in [1, \infty)}$,
of unitaries in $\mathbb{C}1 + \K_{\Mul}$, with
$u_1 = 1_{\Mul}$, such that 
for all $a \in \A$,
$$u_t \phi(a) u_t^* \rightarrow \psi(a) \makebox{ in norm.}$$ 
\end{lem}

\begin{proof}

Before beginning, we recall that  $\pi_B : \Mul(\B) \rightarrow
\C(\B)$ and $\pi_M : \Mul \rightarrow \Mul/\K_{\Mul}$ denote the relevant
quotient maps.  Thus, $(\pi_B \circ \phi)(\A)'$ denotes the commutant
of $(\pi_B \circ \phi)(\A)$ in $\C(\B)$; and 
$(\pi_M \circ \phi)(\A)'$ denotes the commutant of 
$(\pi_M \circ \phi)(\A)$ in $\Mul/\K_{\Mul}$.

The first part of the proof is similar to that of \cite{LoreauxNgSutradhar2}
Theorem 2.6.

Now suppose that the hypotheses (including statements (1) to (5)) are satisfied.

Suppose that both $\phi$ and $\psi$ have large complement.   
  Thus, $\phi(\A)^{\perp} \subset \Mul(\B)$ contains a projection which is Murray--von Neumann equivalent
  to $1_{\Mul(\B)}$, and by
  \cite{ElliottKucerovsky} section 16, page~402 and \cite{GabeEK},
  the map $\phi^+ :
  \A^+ \rightarrow \Mul(\B)$ given by
  $\phi^+ |_{\A} = \phi$ and
  $\phi^+(1) = 1$ is a unital absorbing trivial extension (i.e., $\pi_{\B} \circ \phi^+$ is unitally absorbing).
  The same holds for $\psi$ and $\psi^+$.
  Moreover, $(\phi^+,\psi^+)$ is a generalized homomorphism.  (See
\cite{JensenThomsenBook} Chapter 4 for the generalized homomorphism
picture of KK.)   
  Additionally, $[\phi^+,\psi^+] = 0$ in $KK(\A^+, \B)$
because a homotopy of generalized homomorphisms $(\phi_s,\psi_s)$ between $(\phi,\psi)$ and $(0,0)$ lifts to a homotopy $(\phi^+_s,\psi^+_s)$, and hence $[\phi^+,\psi^+] = [0^+,0^+] = 0$ in $KK(\A^+, \B)$.
  Thus, we may assume that
  $\A$ is unital and $\phi$ and $\psi$ are unital *-monomorphisms.

By \cite{LoreauxNg} Lemma 3.3, there exists a norm continuous path  
$\{ u_{0,t} \}_{t \in [0, \infty)}$ of unitaries in $\Mul(\B)$ such
that $\{ u_{0,t} \}_{t \in [0, \infty)}$ witnesses that 
$$\phi \sim_{asymp, \B} \psi.$$ 

  It is trivial to see that this implies that
  \begin{equation*}
    [\phi, u_{0,0}\phi u_{0,0}^*] = [ \phi, \psi ] = 0,
  \end{equation*}
  and that $\pi_B(u_{0,t}) \in (\pi_B \circ \phi)(\A)'$ for all $t$.

  It is well-known that we have
  a group isomorphism $KK(\A, \B) \rightarrow KK_{Hig}(\A, \B) : [\phi, \psi]
  \rightarrow [\phi, \psi, 1]$.
  (See \cite{HigsonKK} Lemma 3.6.  Here, $KK_{Hig}$ is the version of 
KK-theory presented in \cite{HigsonKK} section 2.)
  Hence, $[\phi, u_{0,0} \phi u_{0,0}^*, 1] = 0$ in $KK_{Hig}(\A, \B)$.
 Hence, by \cite{HigsonKK} Lemma 2.3, $[\phi, \phi, u_{0,0}^*] = 0$
  in $KK_{Hig}(\A, \B)$.

By a unital version of Paschke duality (see \cite{LoreauxNgSutradhar2} Proposition
2.5),  
  there is a group isomorphism $K_1((\pi_B \circ \phi)(\A)') \rightarrow 
KK_{Hig}(\A, \B)$
  which sends $[\pi_B(u_{0,0})]$ to $[\phi, \phi, u_{0,0}^*]$.
  Hence, $[\pi_B(u_{0,0})] = 0$ in $K_1((\pi_B \circ \phi)(\A)')$

By Lemma \ref{lem:MultiplierAlgebraOfSubalgebraOfVNA}, the inclusion
$\B \hookrightarrow \K_{\Mul}$ induces *-homomorphisms
$\Mul(\B) \rightarrow \Mul$ and $\C(\B) \rightarrow \Mul/\K_{\Mul}$,
which in turn induce a *-homomorphism 
$(\pi_{\B} \circ \phi)(\A)' \rightarrow (\pi_M \circ \phi)(\A)' \subset
\Mul/\K_{\Mul}$.  
Hence, $[\pi_M(u_{0,0})] = 0$ in $K_1((\pi_M  \circ \phi)(\A)')$.
By Theorem \ref{thm:PaschkeDualWellDefined}, 
$(\pi_M  \circ \phi)(\A)'$ is $K_1$-injective,
and hence, $\pi_M(u_{0,0}) \sim_h 1$ in $U((\pi_M \circ \phi)(\A)')$  
($\subset \Mul/\K_{\Mul}$).

We claim that 
\begin{equation} \label{equ:GetStrongAUE2} u_{0,0} \sim_h 1 \makebox{ in }
 U(\pi_M^{-1}((\pi_M \circ \phi)(\A)')) \makebox{ (}\subset \Mul \makebox{).}
\end{equation} 
Here is a short proof:
Since $\pi_M(u_{0,0}) \sim_h 1$, by 
\cite{WeggeOlsen} Corollary 4.3.3, 
we can find a unitary $w \in \pi_M^{-1}((\pi_M \circ \phi)(\A)')$ with   
\begin{equation} \label{equ:GetStrongPAUE}
w \sim_h 1 \makebox{ in  } U(\pi_M^{-1}((\pi_M \circ \phi)(\A)'))
\end{equation} 
such that $\pi_M(w) = \pi_M(u_{0,0})$.   So $u_{0,0} \in w + \K_{\Mul}$.
So $u_{0,0} w^* \in 1 + \K_{\Mul}$.  Define $v =_{df} u_{0,0} w^*$.
Then $v \in 1 + \K_{\Mul}$ and $u_{0,0} = v w$.  
Since the unitary group of $1 + \K_{\Mul}$ is path-connected,
$v \sim_h 1$ in $U(1 + \K_{\Mul}) \subset 
U(\pi_M^{-1}((\pi_M \circ \phi)(\A)'))$.  From this and 
(\ref{equ:GetStrongPAUE}), we get (\ref{equ:GetStrongAUE2}) as required.

By Lemma \ref{lem:PAsympUEReduction}, we can find a separable simple
stable
C*-subalgebra $\B_1 \subset \K_{\Mul}$ which satisfies statements (a) to (d),
as well as statement (e) ii., of the conclusion of Lemma 
\ref{lem:PAsympUEReduction} (note that we use (\ref{equ:GetStrongAUE2}) and
that
$\{ u_{0,t} \}_{t \in [0,\infty)}$ gives a path in $\Mul$).  
Hence, we can find a norm-continuous path 
$\{ u_t \}_{t \in [0, \infty)}$, of unitaries in $\Mul(\B_1)$, such that
for all $a \in \A$,
(as maps into $\Mul(\B_1)$)
$$u_t \phi(a) u_t^* - \psi(a) \in \B_1 \makebox{ for all }t,
\makebox{ }\| u_t \phi(a) u_t^* - \psi(a) \| \rightarrow 0,
\makebox{  and  } 
u_0 = 1.$$

  Now for all $t \in [0,\infty)$, let $\alpha_t \in Aut(\phi(\A) + \B_1)$ be given
  by $\alpha_t(x) =_{df} u_t x u_t^*$ for all $x \in \phi(\A) + \B_1$.
  Thus, $\{ \alpha_t \}_{t \in [0, \infty)}$ is a uniformly continuous path of
  automorphisms of $\phi(\A) + \B_1$ such that $\alpha_0 = id$.
  Hence, by \cite{DadarlatEilers} Proposition 2.15 (see also
  \cite{LinStableAUE} Theorems 3.2 and 3.4),
  there exists a (norm-) continuous path $\{ v_t \}_{t\in [0, \infty)}$ of unitaries
  in $\phi(\A) + \B_1$ such that
  $v_0 = 1$ and
  $\| v_t x v_t^* -  u_t x u_t^*  \| 
\rightarrow 0$ as $t \rightarrow \infty$
  for all $x \in \phi(\A) + \B_1$.

  We now proceed as in the last part of the proof of 
\cite{DadarlatEilers} Proposition 3.6 Step 1 (see
  also the proof of \cite{LinStableAUE} Theorem 3.4).
  For all $t \in [0, \infty)$, let $a_t \in \A$ and $b_t \in \B_1$ be the
unique elements such that
  $v_t = \phi(a_t) + b_t$.
Uniqueness follows from the fact that $\pi \circ \phi$ is injective.  
Moreover, by this uniqueness, since $v_0 = 1_{\Mul}$, 
\begin{equation} \label{equ:a0b0} 
a_0 = 1_{\A}\makebox{  and  } b_0 = 0. \end{equation}     
Again, 
since $\pi \circ \phi$ is injective, we have that for all $t$,
  $a_t$ is a unitary in $\A$, and hence, $\phi(a_t)$ is a unitary 
in $\phi(\A) + \B_1$.
  Note also that since $\pi \circ \phi = \pi \circ \psi$ and both maps are
  injective, $\| a_t a a_t^* - a \|  \rightarrow 0$ 
as $t \rightarrow \infty$ for
  all $a \in \A$.
  For all $t$, let $w_t =_{df}  v_t \phi(a_t)^* \in 1 + \B_1$.
Note that since $v_0 = 1$ and by (\ref{equ:a0b0}), we have that
$w_0 = 1$.   
Also, $\{ w_t \}_{t \in [0, \infty)}$ is a norm continuous path of unitaries in
  $\mathbb{C} 1 + \B_1 \subset \mathbb{C}1 + \K_{\Mul}$, and for all $a \in \A$,
  \begin{align*}
    \| w_t \phi(a) w_t^* - \psi(a) \| 
    &\leq \| w_t \phi(a) w_t^* -  v_t \phi(a) v_t^*  \| \\
    &\qquad+ \|  v_t \phi(a) v_t^*  - u_t \phi(a) u_t^* \| \\
    &\qquad+ \| u_t \phi(a) u_t^* - \psi(a) \| \\
    &= \|  v_t\phi( a_t^* a a_t - a) v_t^*  \| \\
    &\qquad+ \|  v_t \phi(a) v_t^*  - u_t \phi(a) u_t^* \| \\
    &\qquad+ \| u_t \phi(a) u_t^* - \psi(a) \| \\
    &\rightarrow 0. \qedhere
  \end{align*}
\end{proof}

We now introduce our KK invariant.
\begin{df} \label{df:KKInvariant}
Let $\A$ be a separable nuclear C*-algebra, and $\Mul$ be a semifinite factor
with separable predual.  Let $\phi, \psi : \A \rightarrow \Mul$
be essential trivial extensions such that 
$\phi(a) - \psi(a) \in \K_{\Mul}$ for all $a \in \A$. 
We define
\[
[\phi, \psi]_{CS} =_{df} [\pi(\{ (1-t) \phi(\cdot) + t \psi(\cdot) \}_{t
\in (0,1)})] \in KK(\A, \C(S \K_{\Mul})).  
\]
\end{df}

\begin{thm} \label{thm:MainResultEasyDirection}
Let $\Mul$ be a semifinite factor with separable predual, and let $\A$
be a separable nuclear C*-algebra.  Let $\phi, \psi : \A \rightarrow \Mul$
be essential trivial extensions, with $\phi(a) - \psi(a) \in \K_{\Mul}$
for all $a \in \A$, and such that
either both $\phi$ and $\psi$ are unital or both have large complement.

If $\phi \sim_{pasymp} \psi$ then $[\phi, \psi]_{CS} = 0$ in $KK(\A, \C(S 
\K_{\Mul}))$. 
\end{thm}

\begin{proof}
By Lemma \ref{lem:PAsympUEReduction}, we can find a separable simple
stable C*-algebra $\B$ such that conditions (a) to (d), as well as
condition (e) i., of the conclusion of Lemma \ref{lem:PAsympUEReduction}
are satisfied.  By Lemma \ref{lem:EKAbsorb} and item (d) above, the
trivial extensions $\phi, \psi : \A \rightarrow \Mul(\B)$ are 
both absorbing.   Since $\phi \sim_{pasymp} \psi$ (as maps into 
$\Mul(\B)$), by \cite{DadarlatEilers} Lemma 3.3 (and also 
\cite{HigsonKK} Lemma 3.6), $[\phi, \psi] = 0$ in $KK(\A, \B)$.
Hence, by Lemma \ref{lem:LinDE}, we can find a norm-continuous path
$\{u_t \}_{t \in [0, 1)}$, of unitaries in $\mathbb{C}1 + \K_{\Mul}$,
with $u_0 = 1$, such that for all $a \in \A$, (as maps into $\Mul$)
$$u_t \phi(a) u_t^* \rightarrow \psi(a) \makebox{ as } t \rightarrow 1.$$
Recalling that $\Mul(S \K_{\Mul}) = C_b((0,1), \Mul)_{stri}$ 
(the bounded strictly continuous functions from $(0,1)$ to $\Mul$),
let $w \in \Mul(S \K_{\Mul})$ be the unitary given by
$w_t =_{df} u_{1-t}$ for all $t \in (0,1)$.  
Hence, 
for all $a \in \A$, as maps into $\Mul(S \K_{\Mul})/ S \K_{\Mul}$,  
$$\pi(w) \pi(\{ (1-t)\phi(a) + t \psi(a) \}_{t \in (0,1)}) \pi(w)^*
= \pi(\{ (1-t) \psi(a) + t \psi(a) \}_{t \in (0,1)})
= \pi(\psi(a)).$$
Hence,
$$[\phi, \psi]_{CS} = [\pi(\{ (1-t) \phi(\cdot) + t \psi(\cdot) \}_{t \in 
(0,1)})] = [\pi \circ \psi] = 0 \makebox{  in  } KK(\A, \C(S \K_{\Mul})).$$  
 
\end{proof}

We require a terminology which will only be used in the next             
two results. Let $\A, \C$ be C*-algebras, with $\C$ unital,
and let $\phi : \A \rightarrow \C$ be a *-homomorphism.
Then $\phi$ is said to be \emph{strongly $O_{\infty}$-stable} 
(see \cite{GabeMemoirs} (1.1)) if there exist bounded continuous 
paths $\{ S^{(j)}_t \}_{t \in [0, \infty)}$ in $\C$ (for $j = 0, 1$) 
such that for all $a \in \A$ and $j, k  = 0, 1$,    
\begin{equation}
\lim_{t \rightarrow \infty} \| S^{(j)}_t \phi(a) - \phi(a) S^{(j)}_t \|
=0 \makebox{  and  }  \lim_{t \rightarrow \infty}
\| ((S^{(j)}_t)^* S^{(k)}_t - \delta_{j,k})\phi(a) \| = 0.   
\end{equation}
In the above, 
$\delta_{j,k} = \begin{cases} 1 & j = k \\ 0 & j \neq k\end{cases}$ (i.e.,
$\delta_{j,k}$ is the Kronecker $\delta$ symbol).

\begin{lem} \label{lem:SOinftyStable}
Let $\Mul$ be a semifinite factor with separable predual, and let $\A$
be a separable nuclear C*-algebra.  Let $\phi, \psi : \A \rightarrow \Mul$
be essential trivial extensions, with $\phi(a) - \psi(a) \in \K_{\Mul}$
for all $a \in \A$, and such that
either both $\phi$ and $\psi$ are unital or both have large complement.

Then the *-homomorphism
$\pi(\{ (1-t) \phi + t \psi \}_{t \in (0,1)}) : \A \rightarrow
\C(S\K_{\Mul})$ is strongly $O_{\infty}$-stable.   
\end{lem}

\begin{proof}[Quick sketch of proof.]  
We may assume that $\phi$ and $\psi$ both are unital (the
proof for the large complement case is essentially the same).

Apply Lemma \ref{lem:PAsympUEReduction} to $\A, \Mul, \phi, \psi$
to get a separable simple stable C*-subalgebra $\B \subseteq \K_{\Mul}$
that satisfies statements (a) to (d) of the conclusion of Lemma
\ref{lem:PAsympUEReduction}.
Since $\B \cong \B \otimes \K$, we may work with $\B \otimes \K$ instead
of $\B$ (and view $\B \otimes \K \subseteq \K_{\Mul}$).

Let $\phi_0 : \A \rightarrow \mathbb{B}(l_2)$ be any (strongly) unital
 trivial 
essential extension.  Then
$\phi_1 =_{df} 1 \otimes \phi_0 : \A \rightarrow 1_{\Mul(\B)} 
\otimes \mathbb{B}(l_2) \subset \Mul(\B \otimes \K)$ 
is a trivial essential extension
with large complement.
By Lemma \ref{lem:PAsympUEReduction} statement (d) and
Lemma \ref{lem:EKAbsorb}, 
$\phi$, $\psi$ (as maps into $\Mul(\B \otimes \K)$)
and $\phi_1$ are all (unitally) absorbing extensions of $\B \otimes \K$.
Hence, 
by \cite{LoreauxNg} Lemma 3.3 
(see also \cite{LeeProperAUE} Theorem 2.5),  
 let $\{ U_t \}_{t \in [0, 1)}$ and
$\{ W_t \}_{t \in (0, 1]}$ be norm continuous paths   
of unitaries in $\Mul(\B \otimes \K)$ such that for all $a \in \A$,
\begin{eqnarray} \label{equ:Paul}
& & U_t \phi(a) U_t^* - \phi_1(a) \in \B \otimes \K \makebox{ }\forall t,  
\makebox{ and  } \| U_t \phi(a) U_t^* - \phi_1(a) \| \rightarrow 0 
\makebox{ as }
t \rightarrow 1-, \makebox{ and }\\  
\label{equ:Kong}
& & 
W_t U_0 \psi(a) U_0^* W_t^* - \phi_1(a) \in \B \otimes \K \makebox{ } \forall t,
\makebox{ and } \| W_t U_0 \psi(a) U_0^* W_t^* - \phi_1(a) \| \rightarrow 0 \makebox{ as } t \rightarrow 0+.  
\end{eqnarray}  

Note that since $\phi(a) - \psi(a) \in \B \otimes \K$ 
for all $a \in \A$, the above implies
that for all $a \in \A$,  
\begin{eqnarray}  \label{equ:Teresa} 
& & W_t U_t \psi(a) U_t^* W_t^* - \phi_1(a),   W_t U_t \phi(a) U_t^* W_t^* 
- W_t \phi_1(a) W_t^* \in \B \otimes \K \makebox{ for all }t, 
\makebox{  and  }\\  
\label{equ:Margaret} 
& & 
\| W_t U_t \psi(a) U_t^* W_t^* - \phi_1(a) \| \rightarrow 0 \makebox{ as }
 t \rightarrow 0+, \makebox{ and }\\  
\label{equ:Vita}
& & \| W_t U_t \phi(a) U_t^* W_t^* 
- W_1 \phi_1(a) W_1^* \| \rightarrow 0 
\makebox{ as } t \rightarrow 1-.         
\end{eqnarray}

For $j = 0,1$, let $V_j \in \Mul(\B) \otimes 1_{\mathbb{B}(l_2)} 
\subseteq \Mul(\B \otimes \K)$ be an isometry such that for all $j,  k$,
\begin{equation} \label{Vita}
V_j^* V_k = \delta_{j,k}, \makebox{ and  } 1_{\Mul(\B \otimes \K)}
\sim 1_{\Mul(\B \otimes \K)} - V_0 V_0^* - V_1 V_1^* \makebox{ in }
\Mul(\B) \otimes 1_{\mathbb{B}(l_2)}.   
\end{equation}
By modifying the paths $\{ U_t \}_{t \in (0,1)}$ and $\{ W_t \}_{t \in (0,1)}$
if necessary, we may assume that
for all $s, t \in [\frac{1}{3}, \frac{2}{3}]$,  $U_s = U_t$ and $W_s = W_t$.

Now, for all $j$,
 since $\phi_1(a) V_j = V_j \phi_1(a)$ for all $a$,  by (\ref{equ:Teresa}),
for all $t$,
$V_j \in \pi_{\B \otimes \K}^{-1}(\pi_{\B \otimes \K} (W_t U_t \psi(\A) U_t^* W_t^*)')$, and    
$W_t V_j W_t^* \in \pi_{\B \otimes \K}^{-1}(\pi_{\B \otimes \K}  
(W_t U_t \phi(\A) U_t^* W_t^*)')$.                            
Note that since $\pi_{\B \otimes \K} \circ \phi = \pi_{\B \otimes \K}
\circ \psi$, $\D_t =_{df} \pi_{\B \otimes \K}^{-1}(
\pi_{\B \otimes \K} (W_t U_t \psi(\A) U_t^* W_t^*)')
= \pi_{\B \otimes \K}^{-1}(
\pi_{\B \otimes \K} (W_t U_t \phi(\A) U_t^* W_t^*)')$ for all $t$.  

Now let $p, q, p', q' \in \D_{\frac{1}{3}}$   
be the projections that are given by $p =_{df} V_0 V_0^*$, $q =_{df} V_1 V_1^*$,
$p' =_{df} W_{\frac{1}{3}} V_0 V_0^* W_{\frac{1}{3}}^*$, and 
$q' =_{df} W_{\frac{1}{3}} V_1 V_1^* W_{\frac{1}{3}}^*$. 
Then in $\D_{\frac{1}{3}}$, $1 \sim p \sim q \sim p' \sim q' \sim
1 - (p + q) \sim 1 - (p' + q')$.  Hence, we can find a unitary  
$Z \in (\D_{\frac{1}{3}})_0$ such that $Zp'Z^* + Z q' Z^* \leq 1 - (p + q)$.
Hence, Let $\{ Z_t \}_{t \in [\frac{1}{2}, \frac{2}{3}]}$ be a norm-continuous path
of unitaries in $\D_{\frac{1}{3}}$ such that 
$Z_{\frac{1}{2}} = Z$ and $Z_{\frac{2}{3}} = 1_{\Mul(\B \otimes \K)}$.    
For $j = 0, 1$, let $\{ \widetilde{S}^{(j)}_t  \}_{t \in (0,1)}$ be the
norm-continuous path of isometries in $\Mul(\B \otimes \K)$        
that is given by 
$$\widetilde{S}^{(j)}_t =_{df}
\begin{cases}
V_j &  t \in (0, \frac{1}{3}]\\
(3 - 6t)^{1/2} V^{(j)} + (6t - 2)^{1/2} Z W_{\frac{1}{3}} V^{(j)} 
W_{\frac{1}{3}}^* & t \in (\frac{1}{3}, \frac{1}{2}] \\    
Z_t  W_{\frac{1}{3}} V^{(j)} 
W_{\frac{1}{3}}^* & t \in (\frac{1}{2}, \frac{2}{3}]\\
W_t V^{(j)}
W_{t}^* & t \in (\frac{2}{3}, 1) 
\end{cases}
$$
Since $\B \subseteq \K_{\Mul}$ and
 $\Mul(\B) \subseteq \Mul$ (unital C*-subalgebra), for $j = 0,1$, 
we may view $\widetilde{S}^{(j)} =_{df} 
\{ \widetilde{S}^{(j)}_t \}_{t \in (0,1)}$ as an isometry in
$\Mul(S \K_{\Mul})$.  Moreover, for $j = 0, 1$, (the constant paths) $S^{(j)}
=_{df} \pi_{\Mul}(\widetilde{S}^{(j)})$  
witness that $\pi(\{ (1-t) \phi + t \psi \}_{t \in (0,1)})$ is 
strongly $O_{\infty}$-stable (as a map into $\C(S \K_{\Mul})$).
\end{proof}

\begin{thm} \label{thm:TotalMainResult} 
Let $\Mul$ be a semifinite factor with separable predual, and let $\A$
be a separable nuclear C*-algebra.  Let $\phi, \psi : \A \rightarrow \Mul$
be essential trivial extensions, with $\phi(a) - \psi(a) \in \K_{\Mul}$
for all $a \in \A$, and such that
either both $\phi$ and $\psi$ are unital or both have large complement.

Then the following statements are equivalent:
\begin{enumerate}
\item $[\phi, \psi]_{CS} = 0$ in $KK(\A, \C(S\K_{\Mul}))$. 
\item $\phi \sim_{pasymp} \psi$. 
\end{enumerate}
\end{thm}

\begin{proof}
The direction (2) $\Rightarrow$ (1) was proven in 
Theorem \ref{thm:MainResultEasyDirection}.

Let us now prove the direction (1) $\Rightarrow$ (2):

Firstly, if $\phi$ and $\psi$ both have large complement, 
we can replace $\A$, $\phi$ and $\psi$ with $\A^+$, $\phi^+$
and $\psi^+$ respectively (where $\phi^+(\alpha 1 + a) = 
\alpha 1_{\Mul} + \phi(a)$ for any $a \in \A$).  Moreover,
since $\phi$ and $\psi$ have large complement, 
$\phi^+, \psi^+ : \A \rightarrow \Mul$ will still be essential
trivial extensions.  Hence, we may assume that $\A$, $\phi$ and $\psi$
are unital.

Thus, the extension
$\pi(\{ (1-t) \phi + t \psi \}_{t \in (0,1)}) : \A \rightarrow 
\C(S \K_{\Mul})$
is unital and full.   
Moreover, by Lemma \ref{lem:SOinftyStable}, 
$\pi(\{ (1-t) \phi + t \psi \}_{t \in (0,1)})$ is also strongly $O_{\infty}$-stable.

Let $\theta : \A \rightarrow 1_{\Mul} \otimes \mathbb{B}(l_2) \subset
\Mul \otimes \mathbb{B}(l_2)$
be a (strongly) unital essential trivial extension. 
Since $\Mul \cong \Mul \otimes \mathbb{B}(l_2)$, 
we identify $\Mul$ with $\Mul \otimes \mathbb{B}(l_2)$ and view  
$\theta$ as a map into $\Mul(S \K_{\Mul})$ (i.e., for all $a \in \A$,
we view $\theta(a)$ as a constant function in $C_b((0,1), \Mul)_{stri} 
\cong \Mul(S \K_{\Mul})$).  Hence,  
$\pi \circ \theta : \A \rightarrow \C(S \K_{\Mul})$  
is a unital, full and strongly $O_{\infty}$-stable *-homomorphism.

Since $[\phi, \psi]_{CS} = 0 = [\pi \circ \theta]$
in $KK(\A, \C(S\K_{\Mul}))$, by \cite{GabeMemoirs} Theorem B, 
$\pi(\{(1-t) \phi + t \psi\}_{t \in (0,1)})$ and $\pi \circ \theta$ 
are asymptotically unitarily equivalent (as maps $\A \rightarrow 
\C(S \K_{\Mul})$).
Hence, we can apply Lemma \ref{lem:ExtReduction} to get a separable
simple stable C*-subalgebra $\B \subset \K_{\Mul}$ which satisfies
the conclusions  of Lemma \ref{lem:ExtReduction}.

Since by Lemma \ref{lem:ExtReduction} item (e), 
$\pi \circ \theta$ and $\pi(\{(1-t) \phi + t \psi\}_{t \in (0,1)})$ are
asymptotically unitarily equivalent as maps into $\C(S \B)$, it follows,
by \cite{GabeLinNg} Lemma 4.3, that we can find a unitary   
$u \in \C(S \B)$ such that $u (\pi \circ \theta) u^* = 
\pi(\{(1-t) \phi + t \psi\}_{t \in (0,1)})$. 
Noting that $\pi \circ \theta$ is a trivial extension (as an extension 
of $S\B$), it follows that
$$[\pi(\{(1-t) \phi + t \psi\}_{t \in (0,1)})] = [\pi \circ \theta] = 0
\makebox{  in  } KK(\A, C(S\B)).$$
Hence, by Proposition \ref{prop:BlackadarProp},  
$$[\phi, \psi] = 0 \makebox{ in } KK(\A, \B).$$   
Hence, by the conclusions of Lemma \ref{lem:ExtReduction},
by Lemma \ref{lem:EKAbsorb} and by Lemma \ref{lem:LinDE}, $\phi$ and $\psi$ are properly asymptotically
unitarily equivalent as maps into $\Mul$.
\end{proof}

\section{Appendix: A KK computation}

\label{sect:Appendix} 

Here, our goal is to give a quick sketch of a proof for 
Proposition \ref{prop:BlackadarProp}, which is stated in 
\cite{BlackadarKTh} 19.2.6  
without proof, and for which we could not find a proof in the standard
textbooks. We will be using both the generalized homomorphism
(see \cite{JensenThomsenBook} Chapter 4) and original Kasparov
(see \cite{JensenThomsenBook} Chapter 2) pictures of KK.\\

\begin{lem} \label{lem:HomotopyInExt}
Let $\A$, $\B$ be separable C*-algebras with $\A$ nuclear and $\B$ simple
purely infinite and stable.
Let $(\phi, \psi), (\phi', \psi')$ be two $KK_h(\A, \B)$-cycles
which are homotopic (i.e., $[\phi, \psi] = [\phi', \psi']$ in
$KK(\A, \B)$;  see \cite{JensenThomsenBook} Chapter 4).

Then $[ \pi(\{ (1-t) \phi + t \psi \}_{t \in (0,1)})]
= [\pi(\{ (1 - t) \phi' + t \psi' \}_{t \in (0,1)})]$ in
$Ext(\A, S \B)$.
\end{lem}

\begin{proof}[Sketch of proof.] 
By replacing $\phi, \psi, \phi', \psi'$ with
$\phi \oplus \sigma, \psi \oplus \sigma, \phi' \oplus \sigma,
\psi' \oplus \sigma$ respectively, 
where $\sigma$ is an absorbing trivial extension,
if necessary, we may assume that $\phi$, $\psi$, $\phi'$ and $\psi'$
are all absorbing (and hence has large complement).

By \cite{LeeProperAUE} Theorem 2.5, 
we can find a norm continuous path $\{ U_t \}_{t \in (0, 1]}$, of
unitaries in $\Mul(\B)$, such that
for all $a \in \A$,
i. $U_t \phi(a) U_t^* - \phi'(a) \in  \B$ for all $t$, and
ii. $\lim_{t \rightarrow 0+} \| U_t \phi(a) U_t^* - \phi'(a) \| = 0$.

From this, and from the assumptions that $(\phi, \psi)$ and $(\phi', \psi')$
are $KK_h$-cycles (i.e., generalized homomorphisms),
it follows that for all $s, t \in (0, 1]$,
for all $a \in \A$, the difference between any two of\\  
$\{ \phi'(a), \psi'(a), U_s \phi(a) U_s^*, U_t \phi(a) U_t^*, 
U_s \psi(a) U_s^*, U_t \psi(a) U_t^* \}$ is an element
of $\B$.
Hence, $(U_1 \phi U_1^*, U_1 \psi U_1^*)$ is a $KK_h(\A, \B)$-cycle which
is equivalent to $(\phi, \psi)$;  and
if we define $\phi_t =_{df} \begin{cases} U_t \phi U_t^* & t \in (0,1]\\
\phi' & t = 0 \end{cases}$, then $\{ (\phi_t, U_1 \phi U_1^*) \}_{t \in [0,1]}$
is a homotopy that witnesses that $(U_1 \phi U_1^*, U_1 \psi U_1^*)$
is equivalent to $(\phi', U_1 \psi U_1^*)$.
Hence, $(\phi, \psi)$, and thus $(\phi', \psi')$, is equivalent to
$(\phi', U_1 \psi U_1^*)$.  Hence, $[\psi', U_1 \psi U_1^*] = 0$ in
$KK(\A, \B)$.
Hence, by \cite{LoreauxNg} Theorem 3.4, we can find a norm continuous path
$\{ V_t \}_{t \in [0,1)}$ of unitaries in $\mathbb{C}1 + \B$, 
such that for all $a \in \A$,
$\lim_{t \rightarrow 1-} \| V_t U_1 \psi(a) U_1^* V_t^* - \psi'(a) \| = 0$.
Moreover, since  $\psi$ has large complement and $\B$ is simple purely       
infinite, we can modify $\{ V_t \}_{t \in [0,1)}$ so that $V_0 \sim_h 1$
in $\mathbb{C}1 + \B$; and thus, we can modify $\{V_t \}_{t \in [0,1)}$
so that $V_0 = 1$.   
Hence,
$\{ W_t =_{df} V_t U_t \}_{t \in (0,1)}$ is a norm-continuous path of
unitaries in $\Mul(\B)$ such that $c_t =_{df}
W_t ((1-t)\phi(a) + t \psi(a)) W_t^*
- ((1-t) \phi'(a) + t \psi'(a)) \in \B$ for all $a \in \A$ and $t \in (0,1)$.
Moreover, $c_t \rightarrow 0$ as $t \rightarrow 0+$, as well
as when $t \rightarrow
1-$.
Hence, $W =_{df} \{ W_t \}_{t \in (0,1)}$ is a unitary in $\Mul(S \B)$ such that
when we conjugate
the extension $\pi\left( \{ (1-t) \phi + t \psi \}_{t \in [0,1]}\right)$
by $\pi(W)$, we get
$\pi\left( \{ (1-t) \phi' + t \psi' \}_{t \in [0,1]}\right)$ 
(as extensions of $S \B$ by $\A$). 
\end{proof}

\subsubsection{} \label{subsubsect:BottClass}
Firstly, we fix some notation to be used only in the
next proof, and
 we define the Bott class $\beta_0 \in KK(\mathbb{C}, S
\hat{\otimes} Cl_1)$,
where $S = C_0(\mathbb{R})$, $Cl_1$ is the Clifford algebra for $n =1$,
and $\hat{\otimes}$ is the graded
tensor product of graded C*-algebras. 
Good references for the Bott class are
 \cite{Echterhoff} p18-19 (``dual-Dirac element" construction) and 
\cite{BlackadarKTh} Exercise 19.9.3 -- especially for more details.
Let $\D =_{df} S \hat{\otimes} Cl_1 = S Cl_1$
given the usual grading (e.g., see \cite{Echterhoff} subsection 3.1;
\cite{BlackadarKTh} section 14).
Let $\widetilde{1} : \mathbb{C} \rightarrow \Mul(\D) : \alpha \mapsto
\alpha 1_{\Mul(\D)}$.  Let $f : \mathbb{R} \rightarrow [-1, 1]$ be a
continuous, odd function such that $f(t) > 0$ for all $t > 0$, and
$lim_{t \rightarrow \infty} f(t) = 1$.
Let $e$ be the standard generator of $Cl_1$ (so $e^2 = <e|e>1 = 1$;  in
the representation $Cl_1 = \mathbb{C} \oplus \mathbb{C}$ (with standard
odd grading), $e = (1, -1)$;  see references mentioned above).

The \emph{Bott class} $\beta_0 \in KK(\mathbb{C}, \D)$ is the class given
by the Kasparov $\mathbb{C}$-$\D$ module
\begin{equation} \label{equdf:BottClass}
(\D, \widetilde{1}, f e).
\end{equation}
By \cite{BlackadarKTh} 19.2.5 and Exercise 19.9.3 (see
also \cite{Echterhoff} Lemma 3.25 and the computations after it 
on pages 21-22), $\beta_0$ is an invertible element of $KK(\mathbb{C}, \D)$, and hence,
$\mathbb{C}$ is KK-equivalent to $\D = SCl_1$.

\begin{lem} \label{lem:KKABToExtASB:Compact}
Let $\A$, $\B$ be separable C*-algebras with $\A$ nuclear and $\B$ stable.

Then there exists a group isomorphism
\begin{equation} \label{equdf1:KKABExtASB}
\Lambda:  KK(\A, \B) \rightarrow Ext(\A, S\B)
\end{equation}
such that if $\phi : \A \rightarrow \B$ is a  *-homomorphism
then $\Lambda$ brings the $KK_h(\A, \B)$-class $[\phi, 0]_{KK}$ to
$[\pi(\{ (1-t) \phi  \}_{t \in (0,1)})]_{Ext}$.
\end{lem}

\begin{proof}[Sketch of proof.]
We freely use the notation from (\ref{subsubsect:BottClass}).
We also use the notation from \cite{BlackadarKTh} -- especially with respect
to graded tensor products
(see \cite{BlackadarKTh} Chapter 14). 
Finally, we give $\A$ and $\B$ the trivial gradings.

To describe the map $\Lambda$, we firstly take the left $\B$-amplification
of $\beta_0$, which is the class $\beta \in KK(\B, \B \hat{\otimes} \D)$
induced by the Kasparov module $(\B \hat{\otimes} \D, 1_{\B}
\hat{\otimes} \widetilde{1},
1_{\B} \hat{\otimes} F_2)$, where $F_2 = f e$.  Since $\beta_0$
is invertible, $\beta$ is also invertible. (See \cite{BlackadarKTh}
17.8.5, and \cite{Echterhoff} P15 after Remark 3.16.)

The group isomorphism
 $\Lambda : KK(\A, \B) \rightarrow Ext(\A, S \B)$ is given by
\begin{equation}
\label{equdf2:KKABExtASB}
\Lambda =_{df}    \Lambda_3 \circ \Lambda_2 \circ \Lambda_1 \circ \Lambda_0.
\end{equation}
Here,
$\Lambda_0 : KK(\A, \B) \rightarrow KK(\A, \B) : x \mapsto -x$.
 $\Lambda_1 : KK(\A, \B) \rightarrow KK(\A, \B \hat{\otimes} \D)$
is the group isomorphism given by $\times \beta$, i.e., Kasparov
product by $\beta$ on the right.   $\Lambda_2 : KK(\A, \B \hat{\otimes} \D)
\rightarrow kK^1(\A, S \B)$ is the inverse of the group isomorphism
in \cite{JensenThomsenBook} 
Proposition 3.3.6 (where $kK^1$ is defined as in 
\cite{JensenThomsenBook} Definition 3.3.4;
note  that  $KK(\A, \B \hat{\otimes} \D)
= KK^1(\A,  S \B)$ by \cite{BlackadarKTh} Df. 17.3.1). And $\Lambda_3 : kK^1(\A, S \B) \rightarrow
Ext(\A, S\B)$ is the inverse of the
group isomorphism from \cite{JensenThomsenBook} Theorem 3.3.10.

Let $\phi : \A \rightarrow \B$ be a *-homomorphism.
We want to apply $\Lambda$ to $\alpha =_{df}
[\phi, 0] \in KK(\A, \B)$.
$\Lambda_0(\alpha) = [0, \phi]$.
Moving from the Cuntz picture to the Kasparov picture,
$\alpha = [\phi, 0]$ corresponds to the Kasparov $\A$-$\B$ module
$(\B, \phi, 0)$.   Hence, by
\cite{BlackadarKTh}  Proposition 18.7.2 and Example 18.4.2(a),
the Kasparov product
$\alpha \times \beta$
is induced by the Kasparov module
$(\B \hat{\otimes} \D, \phi \hat{\otimes} \tilde{1}, 
1_{\B} \hat{\otimes} F_2)$.
But by \cite{JensenThomsenBook} Theorem 2.2.15, $(-\alpha) \times \beta =
-(\alpha \times \beta)$.  Therefore, by the proof of 
\cite{BlackadarKTh} Proposition
17.3.3 (see also \cite{Echterhoff} Theorem 3.9),
$(-\alpha) \times \beta$ is induced by the Kasparov module
$((\B \hat{\otimes} \D)^{op}, \phi \hat{\otimes} \tilde{1}, - 1_{\B} 
\hat{\otimes} F_2)$.
Recalling that $F_2 = f e$, where, in the representation $Cl_1 = \mathbb{C}
\oplus \mathbb{C}$ (with standard odd grading) $e = (1, -1)$,
we see that $(-\alpha) \times \beta$ is induced by the Kasparov module
$\E =_{df} (((\B \otimes S) \oplus (\B \otimes S))^{op}, diag(\phi
\otimes 1, \phi \otimes 1),
diag(-1 \otimes f, 1 \otimes f))$, where the grading on the module
$((\B \otimes S) \oplus (\B \otimes S))^{op}$ is $(x, y) \mapsto (-y, -x)$.
To simplify notation, we now write $S\B, \phi, f$ in place
of $\B\otimes S, \phi \otimes 1, 1 \otimes f$, respectively.
Now the map $( S\B  \oplus S\B)^{op} \rightarrow
S\B \oplus S\B : (x, y) \mapsto (x, -y)$
(where the the latter space is given
the standard odd grading) is a graded module isomorphism which induces
an isomorphism of Kasparov modules (see \cite{JensenThomsenBook}
Definition 2.1.7)
between $\E$ and
$\E' =_{df}  ( S \B \oplus S\B, diag(\phi, \phi),
diag(-f, f))$.  Hence, $\Lambda_1 \circ \Lambda_0(\alpha)
= \Lambda_1(-\alpha) = [\E']$.

We abbreviate further by writing $\E' = ((S \B)^2, diag(\phi, \phi),
diag(-f, f))$.
Now let $P \in \Mul(S\B)$ be such that $2P -1 = -f$; and hence,
$P = \frac{1-f}{2}$.
By the properties of $f$,
$\lim_{t \rightarrow -\infty} P(t) = 1$ and $\lim_{t \rightarrow
\infty} P(t) = 0$.  Since $\phi(\A) \subseteq \B$, it is not
hard to check that $(P, \phi)$ satisfies the definition of
a $KK^1$ cycle in \cite{JensenThomsenBook} (3.3.1), (3.3.2) and
(3.3.3) (i.e.,  modulo $S \B$,
$P$ is a projection up to multiplying by $\phi(a)$, for $a \in \A$).
Hence, $\Lambda_2 \circ \Lambda_1 \circ \Lambda_1(\alpha)
=  
 \Lambda_2 ([\E']) = [P, \phi] \in kK^1(\A, S\B)$.

Finally,
by the definition of the group isomorphism
in \cite{JensenThomsenBook} Theorem 3.3.10 (see also \cite{JensenThomsenBook}
Lemma
3.3.8, and the proofs), $\Lambda_3([P,\phi])$ is the class of the
 extension
$\pi(P \phi ) : \A \rightarrow \C(S\B)$.
Since $\lim_{t \rightarrow -\infty} P(t) = 1$ and
$\lim_{t \rightarrow \infty} P(t) = 0$, 
$\Lambda_3([P, \phi]) = [\pi(P \phi)] =
[\pi(\{(1-t) \phi \}_{t \in (0,1)})]$.
Hence, $\Lambda(\alpha) = \Lambda_3([P, \phi]) =
[\pi(\{(1-t) \phi \}_{t \in (0,1)})] \in Ext(\A, S\B)$ as
required.
\end{proof}

\begin{lem} \label{lem:KKABToExtASB:PurelyInfinite}
Let $\A, \B$ be separable C*-algebras with $\A$ nuclear and
$\B$ simple purely infinite and stable.
Let $\Lambda :  KK(\A, \B) \rightarrow Ext(\A, S \B)$ be the group
isomorphism from Lemma \ref{lem:KKABToExtASB:Compact}.
(See (\ref{equdf1:KKABExtASB}) and (\ref{equdf2:KKABExtASB}).)

Then for any $KK_h(\A, \B)$-cycle $(\phi, \psi)$,
$$\Lambda([\phi, \psi]) = [\pi(\{ (1-t) \phi + t \psi \}_{t \in (0,1)})].$$
\end{lem}

\begin{proof}
Let $(\phi, \psi)$ be a $KK_h(\A, \B)$-cycle.  By \cite{GabeMemoirs} 
Theorem A,  
since $\B$ is simple purely infinite,
we can find a *-homomorphism $\phi_0 : \A \rightarrow \B$ such that
$[\phi, \psi] = [\phi_0, 0]$ in $KK(\A, \B)$.
By Lemma \ref{lem:KKABToExtASB:Compact}, $\Lambda([\phi_0, 0])
= [\pi(\{ (1-t) \phi  \}_{t \in (0,1)})]$.
By Lemma \ref{lem:HomotopyInExt},
$[\pi(\{ (1-t) \phi  \}_{t \in (0,1)})] = [\pi(\{ (1-t) \phi + t \psi \}_{t
\in (0,1)} )]$. Hence,
$\Lambda([\phi, \psi]) = [\pi(\{ (1-t) \phi + t \psi \}_{t
\in (0,1)} )]$.\\
\end{proof}

\begin{prop} \label{prop:BlackadarProp}
Let $\A$ and $\B$ be separable C*-algebras with $\A$ nuclear and
$\B$ simple and stable.
Let $\Lambda :  KK(\A, \B) \rightarrow Ext(\A, S \B)$ be the group
isomorphism from Lemma \ref{lem:KKABToExtASB:Compact}.
(See (\ref{equdf1:KKABExtASB}) and (\ref{equdf2:KKABExtASB}).)

Then for any $KK_h(\A, \B)$-cycle $(\phi, \psi)$,
$$\Lambda([\phi, \psi]) = [\pi(\{ (1-t) \phi + t \psi \}_{t \in (0,1)})]_{Ext}.$$

As a consequence, we have a group isomorphism
$KK(\A, \B) \rightarrow KK(\A, \C(S\B)) : [\phi, \psi]
\mapsto [\pi(\{ (1-t) \phi + t \psi \}_{t \in (0,1)})]_{KK}$.  
\end{prop}

\begin{proof}[Sketch of proof.] 
By \cite{BlackadarKTh} Proposition 23.10.1, the KK-class of the inclusion map
$\iota : \mathbb{C} \rightarrow O_{\infty}$ witnesses that $\mathbb{C}$
and $O_{\infty}$ are KK-equivalent.
Hence, by \cite{BlackadarKTh} Example 19.1.2 (c), the KK-class
of the
*-homomorphism
$\widetilde{\iota}: id_{\B} \otimes  \iota  : \B \otimes \mathbb{C}
\rightarrow \B \otimes O_{\infty}$ witnesses that $\B$ and
$\B \otimes O_{\infty}$ are KK-equivalent.

Also, since $\widetilde{\iota}$ maps any approximate unit of $\B$ to
an approximate unit of $\B \otimes O_{\infty}$,
$\widetilde{\iota}$ induces *-homomorphisms $S \B \rightarrow S \B \otimes 
O_{\infty}$,  $\Mul(S\B) \rightarrow
\Mul(S \B \otimes O_{\infty})$ and
$\C(S \B) \rightarrow \C(S \B \otimes O_{\infty})$, all of which we also
denote by ``$\widetilde{\iota}$". This in turn induces
a group homomorphism
$\widetilde{\iota}_* : Ext(\A, S \B) \rightarrow
Ext(\A, S \B \otimes O_{\infty})$.
Now 
the map $\Lambda$ (see (\ref{equdf2:KKABExtASB})) 
is natural in the variable $\B$. 
Hence, we have a commuting diagram
\begin{equation} \label{equ:CommutingIsomorphismsKKABExtASB}
\begin{array}{ccc}
KK(\A, \B) & \stackrel{\Lambda}{\rightarrow} & Ext(\A, S \B) \\
\times [\widetilde{\iota}] \downarrow  &     &  \widetilde{\iota}_*
\downarrow\\
KK(\A, \B \otimes O_{\infty}) &
\stackrel{\Lambda_{O_{\infty}}}{\rightarrow} &
Ext(\A, S \B \otimes O_{\infty})
\end{array}
\end{equation}
In the above commuting diagram, all the arrows, except possibly
for $\widetilde{\iota}_*$, are group isomorphisms.
Hence, $\widetilde{\iota}_*$ is a group isomorphism.

Let $(\phi, \psi)$ be a $KK_h(\A, \B)$-cycle.
Then $\times [\widetilde{\iota}]$ brings $[\phi, \psi]$ to
$[\widetilde{\iota} \circ \phi, \widetilde{\iota} \circ \psi]
\in KK(\A, \B \otimes O_{\infty})$.  Since $\B$ is simple purely
infinite, by Lemma \ref{lem:KKABToExtASB:PurelyInfinite},
$\Lambda_{O_{\infty}}([\widetilde{\iota} \circ \phi,
\widetilde{\iota} \circ \psi]) = [\pi(\{ (1-t) \widetilde{\iota}
\circ \phi  + t \widetilde{\iota} \circ \psi \}_{t \in (0,1)})]$.
But $\widetilde{\iota}_* ([ \pi(\{ (1-t)
 \phi  + t \psi \}_{t \in (0,1)})]) =
[\pi(\{ (1-t) \widetilde{\iota}
\circ \phi  + t \widetilde{\iota} \circ \psi \}_{t \in (0,1)})]$.
Hence, since (\ref{equ:CommutingIsomorphismsKKABExtASB}) is
a commuting diagram where all the arrows are isomorphisms,
$\Lambda_{O_{\infty}}([\phi, \psi])
= [ \pi(\{ (1-t)
 \phi  + t \psi \}_{t \in (0,1)})]$, as required.

The last statement follows from the previous statement, and by composing
$\Lambda$ with the group isomorphism in 
\cite{DadarlatAUEAndExtTopology} Proposition 4.2. 
\end{proof}


\begin{thebibliography}{00}


\bibitem{ArvesonDuke}  W. Arveson, Notes on extensions of C*-algebras,
Duke Mathematical Journal, 44 (1977), no. 2, 329-355.    


\bibitem{PhillipsEtAl} M. Benameur, A. Carey, J. Phillips,
A. Rennie, F. Sukochev, and K. Wojciechowski,
An analytic approach to spectral flow in von Neumann algebras,
Analysis, geometry and topology of elliptic operators,
World Sci. Publ., New Jersey, 2006.  


\bibitem{BlackadarKTh}  B. Blackadar,
K-theory for operator algebras.
Second edition.  Mathematical Sciences Research Institute Publications, 5.
Cambridge University Press, Cambridge, 1998.


\bibitem{BlackadarOABook} B. Blackadar,
Operator algebras:  theory of C*-algebras and von Neumann algebras, 
Encyclopedia of Mathematical Sciences, 122, Springer--Verlag,
Berlin Heidelberg New York, 2006.  

\bibitem{BlanchardRohdeRordam} E. Blanchard, R. Rohde and M. Rordam,
Properly infinite $C(X)$-algebras and $K_1$-injectivity,
J. Noncommut.  Geom., 2 (2008), no. 3, 263-282.   



\bibitem{BDF1} L. Brown, R. Douglas and P. Fillmore,  
Unitary equivalence modulo the compact operators and extensions
of C*-algebras, Proceedings of a Conference on Operator Theory in
Dalhousie University, Halifax, Nova Scotia, Springer,
Berlin (1973).   

\bibitem{ChoiEffros} 
M. D. Choi and E. G. Effros,
The completely positive lifting problem for C*-algebras,
Ann. Math. 104 (1976), 585-609.  


\bibitem{CiupercaGiordanoNgNiu} A. Ciuperca, T. Giordano, P. W. Ng and
Z. Niu, Amenability and uniqueness, Adv. Math. 240 (2013), 325-345.   


\bibitem{Cuntz} J. Cuntz, K-theory for certain C*-algebras,
Ann. of Math. 113 (1981), 181-197.  


\bibitem{DadarlatAUEAndExtTopology} M. Dadarlat,
Approximate unitary equivalence and the topology of $Ext(\A, \B)$,
C*-algebras (Munster, 1999), 42-60, Springer, Berlin, 2000.  


\bibitem{DadarlatEilers} M. Dadarlat and S. Eilers, 
Asymptotic unitary equivalence in KK-theory, K-Theory, 23
(2001), no. 4, 305-322.  

\bibitem{DavidsonBook}  K. Davidson, C*-algebras by Example, 
The Fields Institute Monographs, 6. American
Mathematical Society, Providence, RI, 1996.



\bibitem{Echterhoff} S. Echterhoff,
Bivariant KK-theory and the Baum--Connes conjecture,
Chapter 3 in the book:  K-theory for group C*-algebras 
and semigroup C*-algebras (Oberwolfach Seminars, 47), 
Birkhauser, Switzerland, 2017.  A copy is available at
arxiv.org/abs/1703.10912.  

\bibitem{ElliottKucerovsky} G. A. Elliott and Dan Kucerovsky, 
G. A. Elliott and D. Kucerovsky, An abstract 
Brown--Douglas--Fillmore absorption theorem, Pacific Journal of
Mathematics 3 (2001), 1-25.   




\bibitem{GabeEK} J. Gabe, A note on nonunital absorbing extensions,
Pacific Journal of Mathematics, 284 (2016), no. 2, 383-393.  


\bibitem{GabeMemoirs} J. Gabe, 
Classification of $O_{\infty}$-stable C*-algebras, Mem. Amer.
Math. Soc. 293 (2024), no. 1461. 


\bibitem{GabeLinNg} J. Gabe, H. Lin and P. W. Ng, 
Extensions of C*-algebras.  Preprint.  A copy is available at\\ 
arxiv.org/abs/2307.15558.


\bibitem{GiordanoKaftalNg} T. Giordano, V. Kaftal, and P. W. Ng,
Aspects of extension theory. Preprint. 


\bibitem{GiordanoNg} T. Giordano and P. W. Ng,
Some consequences of von Neumann algebra uniqueness,
J. Funct. Anal. 264 (2013), no. 4, 1112-1124.   


\bibitem{HadwinMaShen}
D. Hadwin, M. Ma and J. Shen,
Voiculescu's theorem in properly infinite factors,
J. Funct. Anal. 290 (2026), no. 1, Paper No. 111198.   



\bibitem{HigsonKK} N. Higson, A characterization of KK-theory, 
Pacific Journal of Mathematics, 126 (1987), no. 2, 253-276.  



\bibitem{HjelmborgRordam} J. Hjelmborg and M. Rordam,
On stability of C*-algebras, J. Funct. Anal. 155 (1998), no. 1,
153-170.


\bibitem{JensenThomsenBook}  K. K. Jensen and K. Thomsen,
Elements of KK-theory. Mathematics:  Theory and Applications.
Birkhauser Boston, Inc., Boston, MA, 1991.  viii+202 pp.


\bibitem{KadisonPyth1} R. V. Kadison, The Pythagorean 
Theorem I: the finite case, Proc. Natl. Acad. Sci. USA,
99 (2002), no. 7, 4178-4184.  

\bibitem{KadisonPyth2} R. V. Kadison, The Pythagorean
Theorem II:  the infinite discrete case, Proc. 
Natl. Acad. Sci. USA, 99 (2002), no. 8, 5217-5222.  


\bibitem{KaftalLoreaux} V. Kaftal and J. Loreaux,
Kadison's Pythagorean theorem and essential codimension,
Integ. Equ. Oper. Theory, 87 (2017), no. 4, 565-580.  


\bibitem{KucerovskyFredholmTriples}
D. Kucerovsky, When are Fredholm triples operator homotopic?
Proceedings of the AMS, 135 (2007), no. 2, 405-415.   


\bibitem{LeeProperAUE}  H. Lee, Proper asymptotic unitary equivalence
in KK-theory and projection lifting from the corona algebra,
J. Funct. Anal. 260 (2011), no. 1, 135-145.  

\bibitem{LiShenShi} Q. Li, J. Shen and R. Shi,
A generalization of Voiculescu's theorem for normal operators to semifinite
von Neumann algebras,
Adv. Math. 375 (2020), 107347, 55pp.


\bibitem{LinBook} H. Lin,
	An introduction to the classification of amenable C*-algebras. World Scientific Publishing Company, Inc., River Edge, NJ, 2001.
	



\bibitem{LinStableAUE} H. Lin, Stable approximate unitary equivalence of homomorphisms,
Journal of Operator Theory, 47 (2002), no. 2, 343-378.     


\bibitem{LinSimpleCorona} H. Lin, 
Simple Corona C*-algebras,
Proceedings of the American Mathematical Society,
132 (2004), no. 11, 3215-3224.   


\bibitem{LoreauxArveson}  J. Loreaux, Restricted diagonalization
of finite spectrum normal operators and a theorem of Arveson,
Journal of Operator Theory, 81 (2019), no. 2, 257-272.  


\bibitem{LoreauxNg} J. Loreaux and P. W. Ng,
Remarks on essential codimension,
Integral Equations and Operator Theory,
92 (2020), Paper No. 4, 33 pp.  


\bibitem{LoreauxNgSutradhar1} J. Loreaux, P. W. Ng and
A. Sutradhar, Essential codimension and lifting projections,
Houston Journal of Mathematics, 49
(2023), no. 4, 741-784.    


\bibitem{LoreauxNgSutradhar2} J. Loreaux, P. W. Ng and
A. Sutradhar, $K_1$-injectivity of the Paschke dual algebra
for certain simple C*-algebras. 
J. Math. Anal. Appl. 554 (2026), no. 2, Paper No. 129947, 56pp. 


\bibitem{OlsenPedersen}  C. L. Olsen and G. K. Pedersen,
Corona C*-algebras and their applications to lifting problems,
Math. Scand. 64 (1989), 63-86.  


\bibitem{RordamStable}  M. Rordam, Stable C*-algebras,
Advanced Studies in Pure Mathematics 38, Operator algebras and applications,
Math. Soc. of Japan, Tokyo, 117-199 (2004).   

\bibitem{ThomsenPaschke} K. Thomsen, On absorbing extensions, Proceedings of the American
Mathematical Society, 129 (2001), no. 5, 1409-1419.  


\bibitem{Voiculescu} D. V. Voiculescu, 
A noncommutative Weyl--von Neumann theorem, 
Rev. Roum. Pures Appl., 21 (1976), 97-113.  


\bibitem{WeggeOlsen}  N. E. Wegge--Olsen,
	K-theory and C*-algebras. A friendly approach.
	Oxford Science Publications.  The Clarendon Press, Oxford University
	Press, New York, 1993.



\end{thebibliography}
\end{document}